\documentclass[9pt,twoside,english,reqno,a4paper]{amsart}

\usepackage{listings,graphicx,amsmath,varioref,amscd,amssymb,color,bm,stmaryrd,amsthm,amsfonts,graphics,geometry,latexsym,pgf,pst-all} 
\theoremstyle{plain}
\usepackage{esint}
\usepackage{amsthm}
\usepackage{tikz}
\usepackage{pgfplots}
\usepackage{enumerate}

\usepackage{graphicx}
\usepackage{caption}
\usepackage{subcaption}
\renewcommand*{\thesubfigure}{\roman{subfigure})}

\theoremstyle{plain}
\newtheorem{theorem}{Theorem}[section]
\newtheorem{proposition}[theorem]{Proposition}

\newtheorem{lemma}[theorem]{Lemma}

\usepackage{geometry}
\usepackage[colorlinks=false]{hyperref}

\theoremstyle{definition}
\newtheorem{defin}[theorem]{Definition}

\newtheorem{remark}[theorem]{Remark}

\theoremstyle{remark}

\usepackage{fouriernc}
\usepackage[T1]{fontenc}

\font\manual=manfnt
\newcommand\xqed[1]{%
  \leavevmode\unskip\penalty9999 \hbox{}\nobreak\hfill
  \quad\hbox{#1}}
\newcommand\triang{\xqed{\manual\char'170}}

\numberwithin{equation}{section}

\def\dis{\displaystyle}

\DeclareMathOperator{\R}{\mathbb{R}}
\DeclareMathOperator{\N}{\mathbb{N}}

\newcommand{\car}[1]{\raise1pt\hbox{$\chi$}_{#1}}

\newcommand{\DM }{\mathcal{DM}^\infty }
\def\rn{\mathbb{R}^N}
\def\into{\int_{\Omega}}
\def\dH{\mathrm{d}\mathcal{H}^{N-1}}
\def\ae{\mathrm{a.e.}}
\def\re{\mathbb{R}}
\def\rn{\mathbb{R}^N}

\begin{document}
\title[Sub-supersolution and concave-convex problems]{The Sattinger iteration method for 1-Laplace type problems and its application to  concave-convex nonlinearities}

\author[A. J. Mart\'inez Aparicio]{Antonio J. Mart\'inez Aparicio}
\author[F. Oliva]{Francescantonio Oliva}
\author[F. Petitta]{Francesco Petitta}

\address[Antonio J. Mart\'inez Aparicio]{Departamento de Matem\'aticas,
Universidad de Almer\'ia
	\hfill \break\indent
    Ctra. Sacramento s/n, La Ca\^{n}ada de San Urbano, 04120 Almer\'ia, Spain}
\email{\tt ajmaparicio@ual.es}

\address[Francescantonio Oliva]{
Dipartimento di Scienze di Base e Applicate per l' Ingegneria, Sapienza Universit\`a di Roma
	\hfill \break\indent
	Via Scarpa 16, 00161 Roma, Italy}
\email{\tt francescantonio.oliva@uniroma1.it}
\address[Francesco Petitta ]{Dipartimento di Scienze di Base e Applicate per l' Ingegneria, Sapienza Universit\`a di Roma
		\hfill \break\indent
		Via Scarpa 16, 00161 Roma, Italy}
\email{\tt francesco.petitta@uniroma1.it}


\begin{abstract}
In this paper  we extend the classical  sub-supersolution Sattinger iteration method to $1$-Laplace type boundary value  problems of the form 
\begin{equation*}
	\begin{cases}
		\dis -\Delta_1 u = F(x,u) & \text{in}\;\Omega,\\
		u=0 & \text{on}\;\partial\Omega,
	\end{cases}
\end{equation*}
where $\Omega$ is an open bounded domain of $\rn$ ($N\geq 2$) with Lipschitz boundary and $F(x,s)$ is a Carathe\'{o}dory function. This goal is achieved through a perturbation method that overcomes structural obstructions arising from the presence of the $1$-Laplacian and by proving a weak comparison principle for these problems.
As a significant application of our main result we establish existence and non-existence theorems for the so-called ``concave-convex'' problem involving the $1$-Laplacian as leading term. 
\end{abstract}

\maketitle
 
\tableofcontents

\section{Introduction}

In recent years, many authors have dealt with various aspects of problems involving the $1$-Laplacian operator. One of the most basic model of homogeneous Dirichlet problem associated to this operator reads as 
\begin{equation}
	\label{eq:PbIntro1}
	\begin{cases}
		\dis -\Delta_1 u = F(x,u) & \text{in}\;\Omega,\\
		u=0 & \text{on}\;\partial\Omega,
	\end{cases}
\end{equation}
where $\Omega\subset \R^N$ {$(N\geq 2)$} is a bounded open domain with Lipschitz boundary and $F(x,s)$ is  a Carathe\'{o}dory function for which we do not assume any sign or growth assumption. Here $\displaystyle \Delta_1 u := {\rm div}\left(|\nabla u|^{-1} \nabla u\right)$  represents the 1-Laplacian, i.e. the formal limit as $p\to 1^+$ of the usual $p$-Laplace operator $\displaystyle \Delta_p u := {\rm div}\left({|\nabla u|^{p-2}{\nabla u}}\right)$.

\medskip

The primary goal of this study is to introduce a sub- and supersolution method for \eqref{eq:PbIntro1} in its full generality. Apart from its simplicity, this method, when applicable, has some outstanding properties. For example, it usually allows more flexible requirements on the growth conditions of $F(x,s)$ than other methods, such as those from the calculus of variations. Furthermore, it also gives some ordering properties on the solutions.
\medskip

The sub-supersolution method, also known as the order method or the Sattinger method, is a powerful tool which has been widely used in many problems involving, for instance, the $p$-Laplacian ($p>1$) (\cite{sat,amman,mopo,BDGK}). With these operators, the strategy to prove this result is to add, when necessary, a non-decreasing function $l(x,s)$ on both sides of the equation in order to obtain a non-decreasing right-hand. Then, one can construct an iterative scheme with the sub- or the supersolution as a starting point that leads to a solution of the original problem. However, the case of the $1$-Laplacian is strikingly different from that of the $p$-Laplacian and the sub-supersolution method is far from being a trivial extension.
\medskip

On the one hand, it should be clear that a weak comparison principle is key for constructing an iterative scheme. Nevertheless, the non-uniqueness phenomena for solutions to
\begin{equation}
	\label{eq:PbIntro2}
	\begin{cases}
		\dis -\Delta_1 u = f(x) & \text{in}\;\Omega,\\
		u=0 & \text{on}\;\partial\Omega,
	\end{cases}
\end{equation}
even in the natural space $BV(\Omega)$ precludes the existence of a weak comparison principle. In fact, if $u_1,u_2\in BV(\Omega)$ are such that $-\Delta_1 u_1 \leq -\Delta_1 u_2$ ($+$ boundary data), in general one can not deduce that $u_1\leq u_2$ $\mathrm{a.e.}$ in $\Omega$.
\medskip

On the other hand, another additional difficulty when dealing with $1$-Laplace type problems relies on the fact that some smallness assumptions on the norm of $f$ are required in order to have existence of solutions for~\eqref{eq:PbIntro2}. Specifically, it is necessary that $\|f\|_{L^{N,\infty}(\Omega)} \le (\tilde{\mathcal{S}}_1)^{-1}$, where $\tilde{\mathcal{S}}_1$ denotes the best constant in the Sobolev embedding (see for instance \cite{CT, MST1}). This implies that, during the hypothetical iterative scheme, one would have to ask some smallness conditions on the norm of $F(x,s)$ evaluated in both the sub- and the supersolution and, in general, this is not feasible.
\medskip

Due to the difficulties mentioned above, for the $1$-Laplacian operator the approach of the $p$-Laplacian is not effective at first glance.
\medskip

In the current work, we manage to overcome these difficulties by adding some conditions on $l(x,s)$ that do not imply any restriction on the solution given by the method. Two are the key ingredients: the first one is that, as we shall prove, a weak comparison principle is available for the operator $-\Delta_1 u + l(x,u)$ when $l(x,s)$ is increasing. The second one is that, as we show following \cite{AnDaSe}, no smallness condition on $f\in L^N(\Omega)$ is required to have existence in the Dirichlet problem associated with the equation $-\Delta_1 u + l(x,u) = f$ when $\lim_{s\to\pm\infty} l(x,s) = \pm \infty$ uniformly $\mathrm{a.e.}$ in $x\in \Omega$. To the best of our knowledge, sub-supersolution results have only been established for a very specific case, essentially corresponding to $ F(x,s) = |u|^{q-1}u + \lambda $, with $ q > 1 $ (see \cite{Dem1}).

\medskip

As a second aim and after proving the above mentioned sub-supersolution result, we employ it to study existence and non-existence of solutions for problems whose model is given by
\begin{equation}
	\label{eq:PbIntro4}
	\begin{cases}
		\dis -\Delta_1 u = \frac{\lambda f(x)}{u^\gamma} + g(u) & \text{in}\;\Omega,\\
		u=0 & \text{on}\;\partial\Omega,
	\end{cases}
\end{equation}
where $\gamma>0$, $f\in L^{N,\infty}(\Omega)$ is a non-negative function, $g\colon [0,\infty)\to [0,\infty)$ is a continuous function with $g(0)=0$ and $\lambda>0$ is a parameter. These are the so-called ``concave-convex'' problems in the case of the 1-Laplacian (see~\cite{OrPiS}). We prove both existence for $\lambda$ small and non-existence for $\lambda$ large. 

\medskip
Some remarks on the terminology are in order. When talking about concave-convex problems with the $p$-Laplacian operator {($p>1$)}, the literature usually refers to problems whose model is
\begin{equation}
	\label{eq:PbIntro5}
	\begin{cases}
		\dis -\Delta_p u = h(x,u) & \text{in}\;\Omega,\\
		u=0 & \text{on}\;\partial\Omega,
	\end{cases}
\end{equation}
with $h(x,s) = \lambda s^\alpha + s^\beta$ for $0<\alpha< p-1 < \beta$. The name of these problems is given since, for $p=2$, the nonlinearity becomes the sum of a concave and a convex term. For general $p$, the nonlinearity is the addition of a sublinear and a superlinear term with respect to the growth $p-1$. On the other hand, other types of problems widely studied and that behave similarly are the ``singular convex-convex'' problems. Here, the model nonlinearity is $h(x,s) = \frac{\lambda}{s^\alpha} + s^\beta$ with $\alpha>0$ and $ p-1 < \beta$. Again, the lower order term is the sum of a sublinear and a superlinear term with respect to the growth $p-1$. Both types of problems become~\eqref{eq:PbIntro4} when $p=1$. 

\renewcommand*{\thesubfigure}{\tiny{}}

\begin{figure}[ht]
\captionsetup[subfigure]{justification=centering}
\centering
\begin{subfigure}[t]{.245\textwidth}
  \centering
  \includegraphics[scale=0.8]{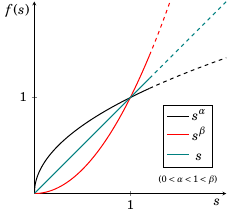}
  \caption{\hspace{4mm} \tiny{concave-convex \\ \hspace{4mm} nonlinearities $(p=2)$}}
\end{subfigure}%
\begin{subfigure}[t]{.245\textwidth}
  \centering
  \includegraphics[scale=0.8]{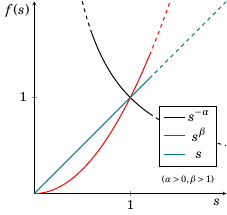}
  \caption{\hspace{4mm} \tiny{singular convex-convex\\
  \hspace{4mm} nonlinearities $(p=2)$}}
\end{subfigure}%
\begin{subfigure}[t]{.245\textwidth}
  \centering
  \includegraphics[scale=0.8]{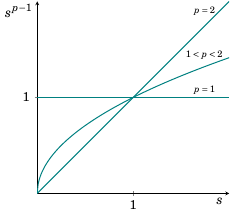}
  \caption{\hspace{4mm} \tiny{$p-1$-linear growth}}
\end{subfigure}
\begin{subfigure}[t]{.245\textwidth}
  \centering
  \includegraphics[scale=0.8]{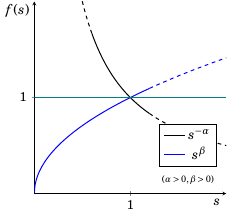}
  \caption{\hspace{4mm}\tiny{sub-superlinear 0-growth}}
\end{subfigure}
\end{figure}%

With respect to the concave-convex problems, the authors in~\cite{BEP} proved using the Sattinger method that a solution to~\eqref{eq:PbIntro5} with $h(x,s) = \lambda s^\alpha + s^\beta$ and $0<\alpha< p-1 < \beta$ always exists for $\lambda$ small, independently of the domain and the growth of the superlinear term. This fact was somewhat surprising, as it was well known since the works of Pohožaev~\cite{Po}, for $p=2$,  and Pucci and Serrin~\cite{PuSe}, for $1<p<N$,  that for starshaped domains a solution to~\eqref{eq:PbIntro5} can not exist for any $\lambda>0$ if $h(x,s) = \lambda s^\beta$ with $\beta\geq p^*-1$. Strongly related to~\cite{BEP} is the seminal paper~\cite{ABC} in which, for $p=2$,  the authors presented a complete study of~\eqref{eq:PbIntro5}, including non-existence results for $\lambda$ large and, when $\beta<2^*-1$, multiplicity results for $\lambda$ small. This result was later extended to the $p$-Laplacian case in~\cite{GaPeMa}. For other relevant related works on this problem, we refer to \cite{dgu, BA} and references therein.
\medskip

Regarding the singular convex-convex problems, there is a lot of literature available. Among others, let us mention the paper~\cite{B} where the author uses the sub-supersolution method to find a solution for $\lambda$ small to~\eqref{eq:PbIntro5} with $h(x,s) = \frac{\lambda}{s^\alpha} + s^\beta$ for $\alpha>0$ and $ p-1 < \beta$. Both the sub- and the supersolution are obtained choosing the solutions to problem~\eqref{eq:PbIntro5} with $h(x,s) = \frac{\tilde\lambda}{s^\alpha}$ for two suitable $\tilde\lambda>0$. For further works dealing with multiplicity, we refer the reader to~\cite{CoPa, ArB, ArMM, GiScTa, HiSaSh}.

\medskip 

On the other hand, when $p=1$, the literature concerning problem \eqref{eq:PbIntro4} is less extensive. Indeed for problems involving $1$-Laplacian operator, most of the existing literature deals with the purely singular case $g\equiv 0$ given by
\begin{equation}
	\label{eq:PbIntro6}
	\begin{cases}
		\dis -\Delta_1 u = h(u)f(x) & \text{in}\;\Omega,\\
		u=0 & \text{on}\;\partial\Omega,
	\end{cases}
\end{equation}
with $f\in L^{N,\infty}(\Omega)$ and $h\colon (0,\infty) \to (0,\infty)$ continuous with $h(\infty):= \lim_{s\to \infty} h(s) < \infty$ and possibly unbounded when approaching to $0$. 
\medskip

As for the existence of solutions to \eqref{eq:PbIntro6}, it is usual to work by approximating it with  
$p$-Laplacian type problems, such as 
\begin{equation}
	\label{eq:PbIntro6bis}
	\begin{cases}
		\dis -\Delta_p u = h(u)f(x) & \text{in}\;\Omega,\\
		u=0 & \text{on}\;\partial\Omega.
	\end{cases}
\end{equation}

These latter problems \eqref{eq:PbIntro6bis} have been extensively studied by the community in the past decades. Just to give an idea, when $h(s) = s^{-\gamma}$ for some $\gamma>0$ and $f(x)$ is a smooth positive function, first~\cite{CrRaTa} and then~\cite{LM} exhaustively studied the existence of classical solutions when $p=2$. Important advances were made in~\cite{BO}, where also for $p=2$ the authors began to consider $f(x)$ merely integrable. This study was extended to the $p$-Laplacian case in \cite{DCA}. An important feature of this problem is that, even for $f(x)\equiv 1$, when $\gamma>\frac{2p-1}{p-1}$ weak solutions are not in the finite energy space (i.e. in $W_0^{1,p}(\Omega)$) but only in $W_\mathrm{loc}^{1,p}(\Omega)$. Moreover, solutions are only bounded if $f\in L^m(\Omega)$ with $m>\frac{N}{p}$. For an extensive study of problems as \eqref{eq:PbIntro6bis}, we refer the interested reader to the recent survey \cite{OPsur}.

\medskip

Coming back to the $1$-Laplacian problem \eqref{eq:PbIntro6}, when $h(s)\equiv 1$, it was shown in~\cite{CT} (see also~\cite{MST1}) that a $BV(\Omega)$ solution can exist only under some smallness assumptions on the datum, namely $\|f\|_{L^{N,\infty}(\Omega)} \le (\tilde{\mathcal{S}}_1)^{-1}$ with $\tilde{\mathcal{S}}_1$ being the best constant of the Sobolev embedding $BV(\Omega) \hookrightarrow L^{N/(N-1),1}(\Omega)$; moreover, if $\|f\|_{L^{N,\infty}(\Omega)} < (\tilde{\mathcal{S}}_1)^{-1}$, the solutions are identically zero.
\medskip

When $h$ is not constant and it is bounded, it was showed in~\cite{DGOP} that a $BV$-solution to~\eqref{eq:PbIntro6} exists if $\|f\|_{L^{N,\infty}(\Omega)} < (\tilde{\mathcal{S}}_1 h(\infty))^{-1}$. Then, if $h(\infty) = 0$, no smallness assumption in the data is required. Furthermore, when $h(s)$ behaves essentially as ${s^{-\gamma}}$ for some $\gamma>0$ near $0$, the authors showed that solution $u\in BV_{\rm loc}(\Omega)$ for the same threshold of $f$ exists satisfying $u^{\max\{1,\gamma\}} \in BV(\Omega)$. This regularity property has recently been improved in \cite{MAOP} showing that $u$ always belongs to $BV(\Omega)$ independently of $\gamma$. Moreover, if $f(x)$ is positive and $h(s)$ goes to infinity at 0 a regularizing effect appears, since all the solutions found in~\cite{DGOP} are positive.

\medskip

Then, once the above setting is well understood, problems as \eqref{eq:PbIntro4} can also be seen as a first natural extension of \eqref{eq:PbIntro4} to the case in which $h$ is actually unbounded at infinity. 
\medskip

Up to our knowledge, problem~\eqref{eq:PbIntro4} has only been studied in~\cite{OrPiS}, where the aim of the authors was to show multiplicity of solutions for $\lambda$ small. Due to the methods employed (essentially variational arguments) they have to add some conditions on the growth of the nonlinearities. In order to be concrete, they consider $\gamma<1$, $f(x)=1$,  and $g(u)=u^q$ with $1<q<\frac{N}{N-1}$.

\medskip 

In the current paper we show that the existence of a global $BV$ solution to~\eqref{eq:PbIntro4} always holds independently of $\gamma$ and $g(s)$ if $\lambda$ is required to be small enough. This existence result is suitably complemented by the non-existence result when $\lambda$ is large.

\medskip

The paper is organized as follows. In Section \ref{sec:prel} some preliminaries are presented with the tools used throughout this work, including a brief review of the 1-Laplace eigenvalue problem and the proofs of the comparison principles needed. Then, in Section \ref{sec:subsuper}, we show the sub-supersolution existence result for the $1$-Laplace Dirichlet problem~\eqref{eq:PbIntro1}.
Next, in Section \ref{sec:existence}, we apply this sub-supersolution method to prove the existence of a solution to the concave-convex problem~\eqref{eq:PbIntro4} for $\lambda$ small and we rely on the 1-Laplace eigenvalue problem to show the non-existence for $\lambda$ small. Finally, Appendices \ref{sec:app1} and \ref{sec:app2} are dedicated to some technical results used in the paper.

\section{Preliminaries}
\label{sec:prel}

Throughout this paper $\Omega$ is an open bounded subset of $\R^N$ ($N\geq 2$) with Lipschitz boundary. Given a set $E\subset \R^N$, $\mathcal{H}^{N-1}(\partial E)$ stands for the $(N-1)$-dimensional Hausdorff measure of the boundary of $E$ while $|E|$ denotes the classical $N$-dimensional Lebesgue measure of $E$.
\medskip

For a fixed $k>0$, $T_{k}$ and $G_{k}$ are the real functions defined by
$$
T_k(s):=\max (-k,\min (s,k)) \qquad \text{and} \qquad G_k(s):=(|s|-k)^+ \operatorname{sign}(s). 
$$
Note in particular that $T_k(s) + G_k(s)=s$ for any $s\in \mathbb{R}$. Moreover, for fixed $\delta>0$ we will also use the following auxiliary  function
\begin{align}\label{not:Vdelta}
V_{\delta}(s):=
\begin{cases}
1 \ \ &s\le \delta, \\
\displaystyle \frac{2\delta-s}{\delta} \ \ &\delta <s< 2\delta, \\
0 \ \ &s\ge 2\delta.
\end{cases}
\end{align}

As for the integrals, we use the following notation 
$$
\int_{\Omega} f := \int_{\Omega} f(x) \ \mathrm{d}x \qquad \text{and} \qquad \int_{\Omega} f\mu := \int_{\Omega} f(x) \ \mathrm{d}\mu,
$$
where $\mu$ is a general measure defined over $\Omega$.
\medskip

Let us explicitly stress that, for the sake of simplicity, by saying that a function $F\colon \Omega\times\R \to \R$ is increasing (non-decreasing) we mean that for $\ae$ $x\in \Omega$ it holds $F(x,s_1)<F(x,s_2)$ ($F(x,s_1)\le F(x,s_2)$) for every $s_1,s_2\in\R$ such that $s_1<s_2$. The reciprocal notation holds for decreasing and non-increasing functions.
\medskip

Finally, let us mention that we denote by $C$ several constants whose value may change from line to line. These values only depend on the data but they do not depend on the indexes of the involved sequences. We underline the use of the standard convention of not relabelling an extracted compact subsequence.

\subsection{Functions of bounded variation}

We denote by $\mathcal{M}(\Omega)$ the space of Radon measures with finite total variation over $\Omega$. Moreover, the space of functions of bounded  variation is denoted by
$$BV(\Omega):= \big\{ u\in L^1(\Omega) : Du \in \mathcal{M}(\Omega)^N \big\},$$
where $Du$ stands for the distributional gradient of $u$. We recall that $BV(\Omega)$ is endowed with the norm
\[
\|u\|=\int_\Omega|Du| + \int_\Omega |u|
\]
  where $\int_\Omega|Du|$ denotes the total variation of the measure $Du$ over $\Omega$; in particular it is a Banach space which is non-reflexive and non-separable.   Another important feature is that every $BV(\Omega)$ function has a  continuous trace operator into $L^1(\partial\Omega)$. This allows to define a norm equivalent to the previous one, namely
\[
\|u\|_{BV(\Omega)}=\int_\Omega|Du| + \int_{\partial\Omega} |u| \ \dH.
\]
Several times it will be used that these norms are lower semicontinuous in $BV(\Omega)$ with respect to the $L^1(\Omega)$ convergence. The same property holds for the functional $u\mapsto \int_\Omega \varphi|Du|$ with $0\leq \varphi\in C_c^1(\Omega)$ fixed.

\medskip
Recall that for $u \in L^1(\Omega)$, $u$ has an approximate limit at $x \in \Omega$ if there exists $\widetilde{u}(x)$ such that 
\begin{equation*}
    \lim_{\rho \downarrow 0} \fint_{B_{\rho}(x)} |u(y) - \widetilde{u}(x)| \, dy = 0,
\end{equation*}
where $\fint_E f = \frac{1}{|E|} \int_E f$ denotes the average integral over $E$. Such points are called \emph{Lebesgue points} of $u$, and the set of these points is denoted by $L_u$. The set where this property does not hold is denoted by $S_u$. This set is a $\mathcal{L}^N$-negligible Borel set \cite[Proposition 3.64]{AFP}.
\medskip

We say that $x$ is an \emph{approximate jump point} of $u$ if there exist $u^+(x) \neq u^-(x)$ and $\nu \in S^{N-1}$ such that
\begin{align*}
    \lim_{\rho \downarrow 0} \fint_{B^+_{\rho}(x, \nu)} |u(y) - u^+(x)| \, dy &= 0, \\
    \lim_{\rho \downarrow 0} \fint_{B^-_{\rho}(x, \nu)} |u(y) - u^-(x)| \, dy &= 0,
\end{align*}
where
\begin{align*}
    B^+_{\rho}(x, \nu) &= \{y \in B_{\rho}(x) : \langle y - x, \nu \rangle > 0\}, \\
    B^-_{\rho}(x, \nu) &= \{y \in B_{\rho}(x) : \langle y - x, \nu \rangle < 0\}.
\end{align*}

The set of approximate jump points is denoted by $J_u$. It is a Borel subset of $S_u$ \cite[Proposition 3.69]{AFP}, and $\mathcal{H}^{N-1}(S_u \setminus J_u) = 0$ if $u \in BV(\Omega)$. Moreover, up to an $\mathcal{H}^{N-1}$-negligible set, $J_u$ is an $\mathcal{H}^{N-1}$-rectifiable set, and an orientation $\nu_u(x)$ is defined for $\mathcal{H}^{N-1}$-almost every $x \in J_u$.
\medskip

The precise representative of $u$  is defined as the $\mathcal{H}^{N-1}$-a.e. finite function 
$u^*\colon \Omega\setminus(S_u\setminus J_u)\to \mathbb{R}$ by 
\begin{equation*}
u^*(x)=
\left\{
\begin{array}{lcr}
\tilde{u}(x)&\hbox{if}&x\in \Omega\setminus S_u,\\
\frac{u^++u^-}{2}&\hbox{if}& x\in J_u.
\end{array}
\right.
\end{equation*}

For further properties of functions of bounded variation, we refer the reader to~\cite{AFP}.
 
\medskip

The classic embedding $W_0^{1,1}(\Omega) \hookrightarrow L^{\frac{N}{N-1}}(\Omega)$ can be extended to the context of $BV(\Omega)$ functions by approximation (\cite{AFP}). This means that
\[
\|u\|_{L^{\frac{N}{N-1}}(\Omega)} \leq \mathcal{S}_1 \|u\|_{BV(\Omega)},
\]
where $\mathcal{S}_1=\big(N\omega_{N}^{\frac{1}{N}}\big)^{-1}$ with $\omega_{N}$ as the volume of the unit sphere of $\rn$ (see for instance \cite{talenti}).
\medskip

The embedding still holds true in the smaller Lorentz space $L^{\frac{N}{N-1},1}(\Omega)$, i.e.
\[
\|u\|_{L^{\frac{N}{N-1},1}(\Omega)} \leq \tilde{\mathcal{S}}_1 \|u\|_{BV(\Omega)},
\]
where $\tilde{\mathcal{S}}_1 = \big( (N-1)\omega_{N}^{\frac{1}{N}} \big)^{-1}$ (see~\cite{alvino}).

\medskip

Here we finally underline that an H\"{o}lder inequality is available for the Lorentz spaces, which are denoted by $L^{p,q}(\Omega)$ with $p>1$ and $q\in[1,\infty]$; in particular, the conjugate space associated with $L^{p,q}(\Omega)$ is given by $L^{p',q'}(\Omega)$. We refer the reader to~\cite{PKOF} for a more in-depth introduction to these spaces.

\subsection{$L^\infty$-divergence-measure vector fields}

When defining the concept of solution for 1-Laplacian type problems, the $L^\infty$-divergence-measure vector fields theory plays a key role. This theory began to be independently developed by~\cite{An} and~\cite{CF}. Let us start defining
\[
\DM(\Omega):= \big\{ z\in L^\infty(\Omega)^N : \operatorname{div}z \in \mathcal{M}(\Omega) \big\}.
\]

First, observe that, if $z\in \DM(\Omega)$, then $\operatorname{div}z$ can be shown to be absolutely continuous with respect to $\mathcal H^{N-1}$ (see~\cite[Proposition~3.1]{CF}). Moreover, a generalized dot product between the vector field $z\in \DM(\Omega)$ and the gradient $Du$ of a function $u\in BV(\Omega)$ can be formulated as
\begin{equation} \label{eq:Pairing}
	\langle(z,Du),\varphi\rangle:=-\int_\Omega u^*\varphi\operatorname{div}z-\int_\Omega
	uz\cdot\nabla\varphi, \,\quad \forall \varphi\in C_c^1(\Omega),
\end{equation}
which is well posed if, for instance, one of following compatibility conditions hold:
\begin{enumerate}[i)]
    \item $z\in \DM(\Omega)$ with $\operatorname{div}z \in L^N(\Omega)$ and $u\in BV(\Omega)$ (see \cite{An});
    \item $z\in \DM(\Omega)$ and $u\in BV(\Omega)\cap L^\infty(\Omega)$ (see \cite{Ca}).
\end{enumerate}
\medskip

Furthermore, under the above conditions, $(z,Du)$ is a Radon measure having finite total variation and for every Borel set $B$ with $B\subseteq U \subseteq \Omega$ ($U$ open) it holds
\begin{equation*} 
    \left| \int_B (z, Du) \right| \leq \int_B |(z, Du)| \leq \|z\|_{L^{\infty}(U)^N} \int_B |Du|.
\end{equation*}

Therefore, $(z,Du)$ is absolutely continuous with respect to the measure $|Du|$. We denote by \(\theta(z, Du, \cdot)\)
the Radon-Nikodym derivative of $(z, Du)$ with respect to $|Du|$, so it holds
\begin{equation*}
	(z, Du) = \theta(z,Du,x) \, {|Du|} \text{ as measures in } \Omega.
\end{equation*}

This derivative behaves well with respect to the composition, as the following result shows.
\begin{lemma}{\cite[Proposition $4.5$]{CrDec}}
	\label{lem:Composition}
	Let $z\in \mathcal{D}\mathcal{M}(\Omega)$ and $u\in BV(\Omega) \cap L^\infty(\Omega)$. Let $\Lambda\colon \R\to\R$ be a non-decreasing locally Lipschitz function. Then
	\[
	\theta(z, D\Lambda(u),x) = \theta(z, Du,x) \quad \text{for } |D\Lambda(u)|\text{-}{\rm a.e.}\ x \in \Omega.
	\]
	As a consequence, $(z,Du)=|Du|$ as measures implies $(z,D\Lambda(u))= |D\Lambda(u)|$ as measures.
\end{lemma}

Since $\Omega$ has a Lipschiz boundary, the outward normal unit vector $\nu(x)$ is defined for $\mathcal H^{N-1}$-almost every $x\in\partial\Omega$. In \cite{An}, it is shown that every $z \in \mathcal{DM}^{\infty}(\Omega)$ possesses
a weak trace on $\partial \Omega$ of the
normal component of $z$ which is denoted by
$[z, \nu]$. This notion of weak trace generalizes the classical one, in the sense that $[z, \nu] = z\cdot \nu$ on $\partial\Omega$ if $z\in C^1(\overline{\Omega}, \R^N)$. Moreover, it verifies
\begin{equation*} 
	\|[z,\nu]\|_{L^\infty(\partial\Omega)}\le \|z\|_{L^\infty(\Omega)^N}.
\end{equation*}

Whenever a compatibility condition holds, a Green formula involving the measure $(z,Du)$ and the weak trace $[z,\nu]$ is verified (see~\cite{An} or~\cite{Ca}), namely
\begin{equation} \label{eq:Green}
	\int_{\Omega} u^* \operatorname{div}z + \int_{\Omega} (z, Du) =
	\int_{\partial \Omega} u[z, \nu] \ \dH.
\end{equation}

\subsection{1-Laplace eigenvalue problem}
\label{sec:eigenvalue}

Let us consider the following minimization problem
\[
h_1(\Omega):=\inf_{E\subset \overline\Omega} \frac{P(E,\R^N)}{|E|},
\]
where $P(E,\R^N)$ is the perimeter in the distributional sense measured with respect to $\R^N$. 
\medskip

Recall that a set $E\subset \R^N$ is said to have finite perimeter in $\R^N$ if the characteristic function $\chi_E$ belongs to $ BV(\R^N)$ and, in such case, $P(E,\R^N)=\int_{\R^N} |D\chi_E|$ (see~\cite{AFP}). The value $h_1(\Omega)$ is called Cheeger constant of $\Omega$ while a set $C\subset \overline{\Omega}$ which minimizes the previous infimum, i.e., a set such that
\[
h_1(\Omega) = \frac{P(C,\R^N)}{|C|} ,
\]
is called Cheeger set for $\Omega$. 
The existence of at least one Cheeger set can be proved while, in general, they are not unique; indeed the uniqueness of the Cheeger set depends strongly on the geometry of the domain. For instance, it is known that convex domains admit a unique Cheeger set. For a gentle introduction to this topic we refer the reader to~\cite{Pa}.
\medskip

Let us now focus on the connection of the Cheeger problem with the first eigenvalue problem for the $1$-Laplacian operator in $\Omega$ with Dirichlet boundary conditions, which is defined as
\begin{equation*}
\lambda_1(\Omega) := \inf_{0\neq \phi\in W^{1,1}_0(\Omega)} \frac{\into |\nabla \phi|}{\into |\phi|} = \inf_{0\neq \phi\in BV(\Omega)} \frac{\into |D\phi| + \int_{\partial\Omega} |\phi| \ \dH}{\into |\phi|}.
\end{equation*}

This infimum is achieved in $BV(\Omega)$ and, moreover, it can be shown that $\lambda_1(\Omega) = h_1(\Omega)$. Thus, among the minimizers (eigenfunctions) one can always find characteristic functions $\chi_C$ associated with any Cheeger set $C$. All the minimizers of~\eqref{eq:PbEigen} satisfy, in a suitable generalized sense (see~\cite[Corollary~4.18]{KaSch} for further details), the Dirichlet problem
\begin{equation*}
	\begin{cases}
		\dis -\Delta_1 \phi  = \lambda_1 \frac{\phi}{|\phi|}  & \text{in}\;\Omega,\\
		\phi=0 & \text{on}\;\partial\Omega.
	\end{cases}
\end{equation*} 
Observe that $\frac{\phi}{|\phi|}$ is just the sign function, which has to be interpreted as a multivalued monotone graph. In~\cite{Dem2}, it is shown that among all the minimizers, there exists a non-negative one  formally satisfying  
\begin{equation*}
	\begin{cases}
		\dis -\Delta_1 \phi  = \lambda_1  & \text{in}\;\Omega,\\
		\phi=0 & \text{on}\;\partial\Omega,
	\end{cases}
\end{equation*} 
in the sense  that there exists some $z_\phi\in L^\infty(\Omega)^N$ with $\|z_\phi\|_{L^\infty(\Omega)^N}\leq 1$ such that  (see also Remark \ref{remphi} below for further insights) 
\begin{equation}\label{eq:PbEigen}
	\begin{cases}
		-\operatorname{div} z_\phi = \lambda_1,\\
		(z_\phi, D\phi) = |D\phi| \text{ as measures in } \Omega,\\
		\phi(1+[z_\phi,\nu])=0 \text{ $\mathcal{H}^{N-1}$-a.e. on } \partial\Omega.
	\end{cases}
\end{equation} 

When $\Omega$ is a self-Cheeger set,  i.e., when $\lambda_1 = \frac{P(\Omega,\re^N)}{|\Omega|}$,  by~\cite[Corollary~4.18]{KaSch} constant positive functions are immediately eigenfunctions solving~\eqref{eq:PbEigen}. For our later purposes, it is important to recall that, in particular, balls are self-Cheeger sets  and thus when $\Omega$ is a ball any positive constant is a solution to~\eqref{eq:PbEigen}. In this specific case, $\lambda_1 = \frac{N}{R}$ where $R$ is the radius of the ball.

\subsection{Comparison principles}

To present a general sub-supersolution method it is crucial to have comparison results between solutions. Following~\cite{Dem1}, we show that under certain conditions some comparison principles still hold for solutions to equations involving the $1$-Laplacian operator. For ease of notation, we introduce the next definition.

\begin{defin}\label{def1h}
    A function $u\in BV(\Omega)$ is said to be almost 1-harmonic in $\Omega$ if there exists some $z\in L^\infty(\Omega)^N$ with $\|z\|_{L^\infty(\Omega)^N}\leq 1$ and $\operatorname{div} z\in L^N(\Omega)$  such that $(z,Du)=|Du|$ as measures in $\Omega$, and it holds ${|u|+ u[z,\nu]=0}$ $\mathcal{H}^{N-1}$-$\ae$ on $\partial\Omega$. 
\end{defin}

\begin{remark}\label{remphi}
Some comments on the previous  definition are needed.  As it is nowadays well known, in order to solve problems as
\begin{equation}
	\label{eq:Pb1}
	\begin{cases}
		 \displaystyle -{\rm div} \left(\frac{D u}{|D u|} \right)= f & \text{in}\;\Omega,\\
		u=0 & \text{on}\;\partial\Omega,
	\end{cases}
\end{equation}
with $f$ in $L^N(\Omega)$, the main issue relies in giving the right meaning to the singular quotient $D u|D u|^{-1}$. Since Anzellotti's work (\cite{An}, see also \cite{ABCM}) this can be done by asking for the existence of a bounded vector field $z$ that mimics the role of $D u|D u|^{-1}$ in \eqref{eq:Pb1}. This is how to interpret the role of $z$ in Definition \ref{def1h}. 

On the other hand, it is also  well known that problem \eqref{eq:Pb1} has a BV-solution only for small data in the sense of the $L^N$-norm (\cite{CT, MST2}); this is a way to think about the terminology {\it almost 1-harmonic}.\triang  
\end{remark}

Here is our  first comparison result. The arguments are essentially the ones given in~\cite[Corollary~3]{Dem1}, where a technical density result is used, namely Proposition~\ref{prop:MP_Dens} below.

\begin{theorem}
\label{th:MP_Nondecr}
    Let $l(x,s)$ be a Carathe\'{o}dory function non-decreasing in $s\in \R$ and let $f_1,f_2\in L^N(\Omega)$. If $u_1,u_2\in BV(\Omega)$ are almost 1-harmonic functions in $\Omega$ such that
    \begin{align}
        \label{eq:Proof_MP_3}
        -\operatorname{div} z_1 + l(x,u_1) &= f_1 \quad \text{ in } \mathcal{D}'(\Omega),\\
        \label{eq:Proof_MP_4}
        -\operatorname{div} z_2 + l(x,u_2) &= f_2 \quad \text{ in } \mathcal{D}'(\Omega)
    \end{align}
    with $f_1<f_2$ in the set $\{u_1>u_2\}$, then $u_1\leq u_2$ $\mathrm{a.e.}$ in $\Omega$.
\end{theorem}

\begin{proof}
Let $u_{1,\varepsilon}$ and $u_{2,\varepsilon}$ be the two sequences given by Proposition~\ref{prop:MP_Dens} to approximate $u_1$ and $u_2$ respectively, and let $\tilde z_1$ and $\tilde z_2$ be the vector fields given by that result. We denote $z_{i,\varepsilon} = |\nabla u_{i,\varepsilon}|^{\varepsilon-1} \nabla u_{i,\varepsilon}$ for $i=1,2$. We first prove that
\begin{equation}
\label{eq:Proof_MP_1}
    -\int_\Omega T_k(u_1-u_2)^+ (\operatorname{div}\tilde z_1-\operatorname{div}\tilde z_2) \geq 0.
\end{equation}

Due to the known inequality
\[
(z_{1,\varepsilon} - z_{2,\varepsilon}) ( \nabla u_{1,\varepsilon} - \nabla u_{2,\varepsilon} )\geq 0,
\]
we deduce that
\[
(z_{1,\varepsilon} - z_{2,\varepsilon}) \nabla T_k(u_{1,\varepsilon} - u_{2,\varepsilon})^+ \geq 0.
\]
Then we get
\begin{equation}
\label{eq:Proof_MP_2}
    -\int_\Omega T_k(u_{1,\varepsilon} - u_{2,\varepsilon})^+ ( \operatorname{div} (z_{1,\varepsilon}) - \operatorname{div} (z_{2,\varepsilon})) = \int_\Omega (z_{1,\varepsilon} - z_{2,\varepsilon}) \nabla T_k(u_{1,\varepsilon} - u_{2,\varepsilon})^+ \geq 0.
\end{equation}

Now, using that $(u_{1,\varepsilon} - u_{2,\varepsilon})^+ \to (u_1-u_2)^+$ in $L^\frac{N}{N-1}(\Omega)$ and that $\operatorname{div}(z_{1,\varepsilon}) - \operatorname{div}(z_{2,\varepsilon}) \to \operatorname{div} \tilde z_1 - \operatorname{div} \tilde z_2$ in $L^N(\Omega)$, we can take the limit in~\eqref{eq:Proof_MP_2} to deduce~\eqref{eq:Proof_MP_1}.

Next, we take $T_k(u_1-u_2)^+$ as test function in~\eqref{eq:Proof_MP_3} and~\eqref{eq:Proof_MP_4} and we subtract the equality resulting from~\eqref{eq:Proof_MP_4} to~\eqref{eq:Proof_MP_3}. Taking into account that $\operatorname{div} \tilde z_i = \operatorname{div} z_i$ for $i=1,2$, we obtain that
\[
-\int_\Omega T_k(u_1-u_2)^+ (\operatorname{div} \tilde z_1-\operatorname{div} \tilde z_2) + \int_\Omega (l(x,u_1)-l(x,u_2)) T_k(u_1-u_2)^+ = \int_\Omega (f_1-f_2) T_k(u_1-u_2)^+.
\]
Using~\eqref{eq:Proof_MP_1}, that $l(x,s)$ is non-decreasing in $s$ and that $f_1<f_2$ in $\{u_1>u_2\}$, we deduce that
\[
\int_\Omega (f_1-f_2)T_k(u_1-u_2)^+ = 0, \forall k>0.
\]

Since $u_1,u_2\in BV(\Omega)$, one can use Lebesgue Theorem to pass to the limit when $k\to\infty$ and obtain
\[
\int_\Omega (f_1-f_2)(u_1-u_2)^+ = 0.
\]
This implies that $f_1=f_2$ in the set $\{u_1>u_2\}$. Since $f_1<f_2$ in $\{u_1>u_2\}$ by hypothesis, we conclude that the measure of the set $\{u_1>u_2\}$ is zero and thus $u_1\leq u_2$ $\mathrm{a.e.}$ in $\Omega$.
\end{proof}

\begin{remark}
Observe that $l(x,s)$ is allowed to be identically zero. For $l(x,s)=0$,  in fact, a general comparison result is shown to hold in~\cite[Corollary~3]{Dem1} under the stronger hypothesis $f_1<f_2$ in $\Omega$. As Theorem~\ref{th:MP_Nondecr} makes clear, it is only necessary to assume $f_1<f_2$ in $\{u_1>u_2\}$. This small detail will be key for comparing solutions in the proof of Proposition~\ref{prop:Approx}. \triang
\end{remark}

Next, we show that the previous comparison principle can be improved if we add an stronger assumption on $l(x,s)$.
\begin{theorem}
\label{th:MP_Incr}
    Let $l(x,s)$ be a Carathe\'{o}dory function increasing in $s\in \R$. If $u_1,u_2\in BV(\Omega)$ are two almost 1-harmonic functions in $\Omega$ such that $l(x,u_1), l(x,u_2)\in L^1(\Omega)$ and
    \begin{align}
        \label{eq:Proof_MP_6}
        -\operatorname{div} z_1 + l(x,u_1) \leq
        -\operatorname{div} z_2 + l(x,u_2) \quad \text{ in } \mathcal{D}'(\Omega),
    \end{align}
    then $u_1\leq u_2$ $\mathrm{a.e.}$ in $\Omega$.
\end{theorem}

\begin{proof}
Arguing as in the proof of Theorem~\ref{th:MP_Nondecr}, we can deduce that
\begin{equation}
\label{eq:Proof_MP_7}
    -\int_\Omega T_k(u_1-u_2)^+ (\operatorname{div}\tilde z_1-\operatorname{div}\tilde z_2) \geq 0,
\end{equation}
where $\tilde z_1$ and $\tilde z_2$ are two vector fields associated to $u_1$ and $u_2$ respectively. Recall that $\operatorname{div}\tilde z_1 = \operatorname{div} z_1$ and $\operatorname{div}\tilde z_2=\operatorname{div} z_2$.

Taking $T_k(u_1-u_2)^+$ for some $k>0$ fixed as test function in~\eqref{eq:Proof_MP_6}, we obtain
\[
-\int_\Omega T_k(u_1-u_2)^+ (\operatorname{div} \tilde z_1-\operatorname{div} \tilde z_2) + \int_\Omega (l(x,u_1)-l(x,u_2)) T_k(u_1-u_2)^+ \leq 0.
\]
Using~\eqref{eq:Proof_MP_7} and that $l(x,s)$ is increasing, we deduce that
\begin{equation*}
\int_\Omega (l(x,u_1)-l(x,u_2)) T_k(u_1-u_2)^+ = \int_{\{u_1>u_2\}} (l(x,u_1)-l(x,u_2)) T_k(u_1-u_2)^+ = 0.
\end{equation*}
Since $l(x,s)$ is increasing in $s\in\R$, the last equality implies that the set $\{u_1>u_2\}$ has measure zero. We conclude that $u_1\leq u_2$ $\mathrm{a.e.}$ in $\Omega$.
\end{proof}

\begin{remark}
Let us stress that, in the previous theorem, the stronger monotonicity assumption on $l(x,s)$ is not only technical. Due to the homogeneity of the operator, in fact, if $l(x,s)$ is only assumed to be non-decreasing with respect to $s$, then  known non-uniqueness outcomes appear and Theorem \ref{th:MP_Incr} can not hold. \triang
\end{remark}

\section{A general existence result via sub-supersolution method}
\label{sec:subsuper}

Let us consider
\begin{equation} \label{eq:PbSS}
	\begin{cases}
		\dis -\Delta_1 u = F(x,u) & \text{in}\;\Omega,\\
		u=0 & \text{on}\;\partial\Omega,
	\end{cases}
\end{equation}
where $F\colon \Omega\times \R\to \mathbb{R}$ is a Carathe\'{o}dory function on which, we underline, no sign condition is imposed. As it is usual in  sub-supersolution existence type results, we assume that there is a Carathe\'{o}dory function $l\colon \Omega\times \R \to \mathbb{R}$ such that
 \begin{equation} \label{eq:hyp_ss_l_incr}
    l(x,s) \text{ is increasing in } s\in \mathbb{R} \text{ and }F(x,s) + l(x,s) \text{ is non-decreasing in } s\in \mathbb{R}.
\end{equation} 
Without loosing generality we can suppose that
\begin{equation}
    \label{eq:hyp_ss_l_inf}
    \lim_{s\to \pm \infty} l(x,s) = \pm \infty \text{ uniformly w.r.t. } x \qquad \text{and} \qquad l(x,s)s \geq 0,\ \forall s\in\R.
\end{equation}
We further assume  that it holds
\begin{equation} \label{eq:hyp_ss_l_reg}
    l(x,s)\in L^N(\Omega) \text{ for all } s\in\R \text{ fixed}.
\end{equation}
We first specify what we mean by a sub- and a supersolution to \eqref{eq:PbSS}.

\begin{defin}  
    A function $u\in BV(\Omega)$ is a \textit{subsolution} (resp. a \textit{supersolution}) to~\eqref{eq:PbSS} if $F(x,u) \in L^N(\Omega)$ and if there exists $z\in \DM(\Omega)$ with $\|z\|_{L^\infty(\Omega)^N}\leq 1$ and $\operatorname{div} z\in L^N(\Omega)$ such that
    \begin{equation}\label{def_distr}
    	-\operatorname{div} z \overset{(\geq)}{\leq}  F(x,u)   \text{ in } \mathcal{D}'(\Omega),
    \end{equation}
     \begin{equation}\label{def_pairing}
 	(z, Du)=|Du| \text{ as measures in } \Omega,
 	\end{equation}   
    \begin{equation}\label{def_boundary}
	|u(x)|+u(x)[z,\nu] (x) =0 \text{ for } \mathcal{H}^{N-1}\text{-a.e. } x \in \partial\Omega.
  	\end{equation}   
     We say that $u$ is a \textit{solution} to \eqref{eq:PbSS} if it is both a sub- and a supersolution to \eqref{eq:PbSS}.
\end{defin}

Now we state and prove the main theorem of this section.

\begin{theorem}
\label{th:SS}
Assume that $F(x,s)$ and $l(x,s)$ verify~\eqref{eq:hyp_ss_l_incr}--\eqref{eq:hyp_ss_l_reg}. Let $w$ be a subsolution to~\eqref{eq:PbSS} and let $v$ be a supersolution to~\eqref{eq:PbSS} such that
\begin{gather}
    \label{eq:hyp_ss_LN}
    l(x,w),\ l(x,v)\in L^N(\Omega).
\end{gather}
If $w\leq v\ \ae$ in $\Omega$, then problem~\eqref{eq:PbSS} has a solution $u\in BV(\Omega)\cap L^\infty(\Omega)$ satisfying that $w\leq u\leq v \ \ae$ in $\Omega$.
\end{theorem}

\begin{remark}
	  Let us stress that assumption \eqref{eq:hyp_ss_l_incr} can be weakened by requiring that the monotonicity holds only for $s \in [w(x), v(x)]$ for almost every $x\in \Omega$. \triang
\end{remark}

\begin{remark}
As we will see, we actually show the existence of a minimal solution $u_1\in BV(\Omega)\cap L^\infty(\Omega)$ to~\eqref{eq:PbSS}, in the sense that if $u$ is a solution to~\eqref{eq:PbSS} such that $w\leq u\leq v$, then $u_1\leq u$. Indeed one can make the same proof using $u$ instead of $v$ as a supersolution to~\eqref{eq:PbSS}. 

On the other hand, if one starts the iteration procedure from above, one can also show the existence of a maximal solution $u_2\in BV(\Omega)\cap L^\infty(\Omega)$ to~\eqref{eq:PbSS}. Therefore, every solution $u$ to~\eqref{eq:PbSS} such that $w\leq u\leq v$ verifies $u_1\leq u\leq u_2$ and, as a consequence, belongs to $L^\infty(\Omega)$. \triang
\end{remark}

\begin{remark}\label{rembv}
    Assumption~\eqref{eq:hyp_ss_l_incr} is equivalent to asking that, for $\ae$ $x\in\Omega$ fixed, $F(x,s)\in BV_{\rm loc}(\re)$. Indeed, observe that, if~\eqref{eq:hyp_ss_l_incr} is verified, then $F(x,s)$ can be written as the difference of two non-decreasing functions, $F(x,s)+l(x,s)$ and $l(x,s)$; then the Jordan decomposition Theorem implies that $F(x,s)\in BV_{\rm loc}(\re)$ for $\ae$ $x\in\Omega$. Conversely, if $F(x,s)\in BV_{\rm loc}(\re)$ for $\ae$ $x\in\Omega$, again the Jordan decomposition Theorem states that $F(x,s)$ can be written as the difference of two non-decreasing functions $m(x,s)$ and $l(x,s)$, namely $F(x,s) = m(x,s) - l(x,s)$. Therefore, $F(x,s) + l(x,s)$ is non-decreasing. \triang
\end{remark}

\begin{proof}[Proof of Theorem \ref{th:SS}]
\mbox{}
The proof is carried on by subsequent steps.

\smallskip

\textbf{Step 1.} Construction of a non-decreasing sequence $u_n$ such that $w\leq u_n\leq v$ a.e. in $\Omega$.

Let $u_0:=w$. We start our iterative scheme by defining $u_1$ as the unique solution to
\begin{equation*}
	\begin{cases}
		\dis -\Delta_1 u_1 + l(x,u_1) = F(x,u_0) + l(x,u_0) & \text{in}\;\Omega,\\
		u_1=0 & \text{on}\;\partial\Omega.
	\end{cases}
\end{equation*} 
The existence of $u_1\in BV(\Omega)\cap L^\infty(\Omega)$ is granted by Theorem~\ref{th:Abs} as $F(x,u_0) + l(x,u_0)\in L^N(\Omega)$ by~\eqref{eq:hyp_ss_LN} and $l(x,s)$ satisfies both~\eqref{eq:hyp_ss_l_inf} and~\eqref{eq:hyp_ss_l_reg}.  Moreover, as $l$ is increasing in $s$ and \eqref{eq:hyp_ss_l_reg} is in force, one gains that $l(x,u_1) \in L^N(\Omega)$.  

Now let us show that $u_0\leq u_1\leq v$ a.e. in $\Omega$. Indeed, as $u_0$ is a subsolution to~\eqref{eq:PbSS}, one has that
\[
-\Delta_1 u_0 + l(x,u_0) \leq -\Delta_1 u_1 + l(x,u_1),
\]
which allows to apply Theorem \ref{th:MP_Incr} to assure that $u_0\leq u_1$ a.e. in $\Omega$. On the other hand, as $u_0=w\leq v$, hypothesis~\eqref{eq:hyp_ss_l_incr} can be applied to obtain $F(x,u_0)+l(x,u_0) \leq F(x,v) + l(x,v)$. Since $v$ is a supersolution to~\eqref{eq:PbSS}, this implies that
\[
-\Delta_1 u_1 + l(x,u_1) \leq -\Delta_1 v + l(x,v), 
\]
and, once more, Theorem~\ref{th:MP_Incr} gives that $u_1\leq v$ a.e. in $\Omega$. A consequence of $u_0\leq u_1\leq v$ and both hypotheses~\eqref{eq:hyp_ss_l_incr} and~\eqref{eq:hyp_ss_LN} is that $F(x,u_1)+l(x,u_1) \in L^N(\Omega)$.

Hence, as applications of Theorem \ref{th:Abs}, once $u_{n-1}$ is defined, one can recursively set $u_n$ as the unique solution to
\begin{equation}\label{eq:Proof_SS_1}
	\begin{cases}
		\dis -\Delta_1 u_n + l(x,u_n) = F(x,u_{n-1}) + l(x,u_{n-1}) & \text{in}\;\Omega,\\
		u_n=0 & \text{on}\;\partial\Omega.
	\end{cases}
\end{equation} 
Following the previous arguments, one can prove by induction that, for every $n\in \N$, function $u_n$ belongs to $BV(\Omega)\cap L^\infty(\Omega)$ and verifies that $w\leq u_{n-1}\leq u_n \leq v$ and that $F(x,u_n)+l(x,u_n) \in L^N(\Omega)$.

Therefore, $u_n$ is a non-decreasing sequence for which we can define $u(x)=\lim_{n\to\infty} u_n(x)$ for almost every $x\in\Omega$. Observe that $w\leq u\leq v$ a.e. in $\Omega$. We will show that $u$ is a solution to~\eqref{eq:PbSS}.

\smallskip

\textbf{Step 2.} Existence of the vector field $z$.

Let $z_n\in \DM(\Omega)$ be the associated vector fields to $u_n$. We recall that $\|z_n\|_{L^\infty(\Omega)^N} \leq 1$ and that $(z_n,Du_n) = |Du_n|$ as measures in $\Omega$. Moreover, the boundary condition is satisfied in this way: $|u_n|+u_n[z_n,\nu] =0$ $\mathcal{H}^{N-1}$-$\ae$ on $\partial\Omega$.

Since $z_n$ is bounded in $L^\infty(\Omega)^N$, then there exists $z\in L^\infty(\Omega)^N$ such that $z_n\rightharpoonup z$ *-weakly in $L^\infty(\Omega)^N$. As the norm is *-weakly lower semicontinuous, then we deduce $\|z\|_{L^\infty(\Omega)^N}\leq 1$.

\smallskip

\textbf{Step 3.} $u_n$ is bounded in $L^\infty(\Omega)$ with respect to $n$.

Firstly let us define
\begin{equation*}
    \tilde{F}(x) := \max\{|F(x,w)+l(x,w)|, |F(x,v)+ l(x,v)|\} \in L^N(\Omega).
\end{equation*} 
Due to hypothesis~\eqref{eq:hyp_ss_l_incr} and since $w\leq u_n\leq v$, we know that
\begin{equation}
    \label{eq:Proof_SS_2}
    |F(x,u_n)+l(x,u_n)| \leq \tilde F(x),\ \forall n\in\N.
\end{equation}
Then, taking $G_k(u_n)\in BV(\Omega)\cap L^\infty(\Omega)$ $(k>0)$  as test function in~\eqref{eq:Proof_SS_1}, using~\eqref{eq:Green} and \eqref{eq:Proof_SS_2}, one gets
\begin{equation}
	\begin{aligned}
    \label{eq:Proof_SS_3}
    &\int_\Omega (z_n,D G_k(u_n)) - \int_{\partial\Omega} G_k(u_n)[z_n,\nu] \ \dH + \int_\Omega l(x,u_n) G_k(u_n) 
    \\
    &= \int_\Omega [F(x,u_{n-1}) + l(x,u_{n-1})] G_k(u_n) 
    \leq \int_\Omega \tilde F(x) |G_k(u_n)|.
    \end{aligned}
\end{equation}
Now observe that $u_n[z_n,\nu] =-|u_n|$ $\mathcal{H}^{N-1}$-a.e. on $\partial\Omega$ implies that $G_k(u_n)[z_n,\nu] =-|G_k(u_n)|$ $\mathcal{H}^{N-1}$-a.e. on $\partial\Omega$ since $G_k(s)s\geq 0$ for all $s\in\R$. Moreover, as $G_k$ is a Lipschitz function, one has that $(z_n, DG_k(u_n)) = |DG_k(u_n)|$ as measures in $\Omega$ by Lemma~\ref{lem:Composition}. Then, from~\eqref{eq:Proof_SS_3}, one deduces
\begin{equation}
    \label{eq:Proof_SS_4}
    \int_\Omega |DG_k(u_n)| + \int_{\partial\Omega} |G_k(u_n)| \ \dH + \int_\Omega l(x,u_n) G_k(u_n) \leq \int_\Omega \tilde F(x) |G_k(u_n)|.
\end{equation}

Let $h>0$ be a number that will be fixed later. Due to hypothesis~\eqref{eq:hyp_ss_l_inf}, there exists some $k>0$ depending on $h$ such that
\begin{equation*}
    \inf_{|s|\in [k,\infty)} |l(x,s)| \geq h.
\end{equation*}
Then, for the third integral of~\eqref{eq:Proof_SS_4} one has
\begin{equation}
    \label{eq:Proof_SS_5}
    h \int_\Omega |G_k(u_n)| \leq \inf_{|s|\in [k,\infty)} |l(x,s)| \int_\Omega |G_k(u_n)| \leq \int_\Omega l(x,u_n) G_k(u_n).
\end{equation}

For the right-hand of~\eqref{eq:Proof_SS_4}, using H\"{o}lder's and Sobolev's inequalities we have
\begin{equation}
    \label{eq:Proof_SS_6}
    \begin{split}
    \int_\Omega \tilde F(x) |G_k(u_n)| &\leq \int_{\{\tilde F\leq h\}} \tilde F |G_k(u_n)| + \int_{\{\tilde F>h\}} \tilde F |G_k(u_n)| \\
    &\leq h \int_\Omega |G_k(u_n)| + \left\|\tilde F \chi_{\{\tilde F>h\}} \right\|_{L^N(\Omega)} \|G_k(u_n)\|_{L^{\frac{N}{N-1}}(\Omega)} \\
    &\leq h \int_\Omega |G_k(u_n)| + \left\|\tilde F \chi_{\{\tilde F>h\}} \right\|_{L^N(\Omega)} \mathcal{S}_1 \left( \int_\Omega |D G_k(u_n)| + \int_{\partial\Omega} |G_k(u_n)| \ \dH \right).
    \end{split}
\end{equation}

Joining~\eqref{eq:Proof_SS_4},~\eqref{eq:Proof_SS_5} and~\eqref{eq:Proof_SS_6} we obtain
\begin{equation}
    \label{eq:Proof_SS_7}
    \left( 1- \left\|\tilde F \chi_{\{\tilde F>h\}} \right\|_{L^N(\Omega)} \mathcal{S}_1 \right) \left( \int_\Omega |DG_k(u_n)| + \int_{\partial\Omega} |G_k(u_n)| \ \dH \right) \leq 0.
\end{equation}

Now, we fix $h$ large enough in order to have
\begin{equation*}
    \left\|\tilde F \chi_{\{\tilde F>h\}} \right\|_{L^N(\Omega)} \mathcal{S}_1 < 1.
\end{equation*}

This allows to deduce from~\eqref{eq:Proof_SS_7} that $\|G_k(u_n)\|_{BV(\Omega)} = 0$ and thus $G_k(u_n)=0$ $\ae$ in $\Omega$. Therefore, we conclude that $\|u_n\|_{L^\infty(\Omega)}\leq k$ for every $n\in\N$. This also shows that $u\in L^\infty(\Omega)$ as $u_n\to u$ almost everywhere in $\Omega$.

\smallskip

\textbf{Step 4.} $u_n$ is bounded in $BV(\Omega)$ with respect to $n$.

Taking $u_n\in BV(\Omega)\cap L^\infty(\Omega)$ as test function in~\eqref{eq:Proof_SS_1}, using~\eqref{eq:Green} and recalling~\eqref{eq:Proof_SS_2}, we get
\begin{equation*}
    \int_\Omega (z_n,Du_n) - \int_{\partial\Omega} u_n[z,\nu] \ \dH + \int_\Omega l(x,u_n) u_n = \int_\Omega [F(x,u_{n-1}) + l(x,u_{n-1})] u_n \leq \int_\Omega \tilde F(x) |u_n|.
\end{equation*}
As $(z_n, Du_n) = |Du_n|$ as measures in $\Omega$ and $u_n[z_n,\nu] = -|u_n|$  $\mathcal{H}^{N-1}$-a.e. on $\partial\Omega$, from the previous one gets
\begin{equation}    \label{eq:Proof_SS_8bis}
	\int_\Omega |Du_n| + \int_{\partial\Omega} |u_n|\ \dH \leq \int_\Omega \tilde F(x) |u_n| \leq  C\int_\Omega \tilde F(x),
\end{equation}
where we have also used \eqref{eq:hyp_ss_l_inf} and the fact that $u_n$ is bounded in $L^\infty(\Omega)$ with respect to $n$. 
Since $\tilde F\in L^N(\Omega)$, from \eqref{eq:Proof_SS_8bis} we conclude that $u_n$ is bounded in $BV(\Omega)$ with respect to $n$. Therefore standard arguments show that $u\in BV(\Omega)\cap L^\infty(\Omega)$ and also that $u_n\to u$ in $L^q(\Omega)$ for every $q<\infty$.

\smallskip

\textbf{Step 5.} Distributional formulation \eqref{def_distr}.

We take $\varphi\in C_c^1(\Omega)$ as test function in~\eqref{eq:Proof_SS_1} obtaining, after using~\eqref{eq:Green}, that
\begin{equation} \label{eq:Proof_SS_8tris}
    \int_\Omega z_n \cdot \nabla \varphi + \int_\Omega l(x,u_n) \varphi = \int_\Omega \left(F(x,u_{n-1}) + l(x,u_{n-1})\right) \varphi,
\end{equation}
and we aim to take $n\to\infty$ into the previous.
The first term of \eqref{eq:Proof_SS_8tris} simply passes to the limit as $z_n\to z$ *-weakly in $L^\infty(\Omega)^N$. For the second and the third term one can apply the Lebesgue Theorem thanks to \eqref{eq:hyp_ss_LN} and \eqref{eq:Proof_SS_2} since $w\le u_n\le v$ in $\Omega$. Therefore, we have shown
\begin{equation*}
    \int_\Omega z \cdot \nabla \varphi = \int_\Omega F(x,u) \varphi,\ \forall \varphi\in C_c^1(\Omega),
\end{equation*}
i.e.,
\begin{equation} \label{eq:Proof_SS_11}
    -\operatorname{div} z = F(x,u) \text{ in } \mathcal{D}'(\Omega).
\end{equation}

\smallskip

\textbf{Step 6.} Role of $z$ given by \eqref{def_pairing}.

Here we take $u_n\varphi\in BV(\Omega)\cap L^\infty(\Omega)$ with $0\leq \varphi\in C_c^1(\Omega)$ as test function in~\eqref{eq:Proof_SS_1}. Using~\eqref{eq:Green} and that $(z_n,Du_n) = |Du_n|$ as measures in $\Omega$, we get
\begin{equation*}
    \int_\Omega |Du_n| \varphi + \int_\Omega z_n \cdot \nabla \varphi \ u_n + \int_\Omega l(x,u_n) u_n\varphi = \int_\Omega (F(x,u_{n-1}) + l(x,u_{n-1})) u_n\varphi,
\end{equation*}
 and we aim to take $n\to\infty$ into the previous. On the first integral we use the lower semicontinuity. On the second one we use that $z_n\to z$ *-weakly in $L^\infty(\Omega)^N$ and that $u_n\to u$ in $L^q(\Omega)$ for any $q<\infty$. Finally, for the remaining terms we can apply once more the Lebesgue Theorem as for the conclusion of Step $5$. Thus, we have proved that
\begin{equation}
    \label{eq:Proof_SS_9}
    \int_\Omega |Du| \varphi \leq \int_\Omega F(x,u) u\varphi - \int_\Omega uz \cdot \nabla \varphi.
\end{equation}
Now observe that, taking $u\varphi$ as test function in~\eqref{eq:Proof_SS_11}, one obtains that
\begin{equation}
    \label{eq:Proof_SS_12}
    -\into u\varphi \operatorname{div} z = \into F(x,u)u \varphi.
\end{equation}
Then, from~\eqref{eq:Proof_SS_9},~\eqref{eq:Proof_SS_12} and~\eqref{eq:Pairing}, we get
\[
\int_\Omega |Du| \varphi \leq \int_\Omega (z,Du)\varphi,\ \forall \varphi\in C_c^1(\Omega) \text{ with }\varphi\ge 0.
\]
As the reverse inequality is immediate since $\|z\|_{L^\infty(\Omega)^N}\leq 1$, we conclude that $(z,Du) = |Du|$ as measures in $\Omega$.

\smallskip

\textbf{Step 7.} Boundary condition \eqref{def_boundary}.

As in Step 4 we take $u_n \in BV(\Omega)\cap L^\infty(\Omega)$ as test function in~\eqref{eq:Proof_SS_1}. Using~\eqref{eq:Green} we obtain
\begin{equation*}
    \int_\Omega (z_n, Du_n) - \int_{\partial\Omega} u_n[z_n,\nu]\ \dH + \int_\Omega l(x,u_n) u_n = \int_\Omega (F(x,u_{n-1}) + l(x,u_{n-1})) u_n.
\end{equation*}

Using that $(z_n, Du_n) = |Du_n|$ as measures in $\Omega$ and that $u_n[z_n,\nu] = -|u_n|$  $\mathcal{H}^{N-1}$-a.e. on $\partial\Omega$, the previous expression takes to
\begin{equation}
    \label{eq:Proof_SS_10}
    \int_\Omega |Du_n| + \int_{\partial\Omega} |u_n|\ \dH + \int_\Omega l(x,u_n) u_n = \int_\Omega (F(x,u_{n-1}) + l(x,u_{n-1})) u_n.
\end{equation}

Next, we want to pass to the limit into \eqref{eq:Proof_SS_10}. For the first two terms on the left-hand of \eqref{eq:Proof_SS_10}, we can use the lower semicontinuity of the norm. In the remaining terms, we can easily apply the Lebesgue Theorem. Hence, we have shown that 
\begin{equation*}
    \int_\Omega |Du| + \int_{\partial\Omega} |u| \ \dH \leq \int_\Omega F(x,u)u.
\end{equation*}
Then it follows from taking $u$ as test function in \eqref{eq:Proof_SS_11} and from \eqref{eq:Green} that it holds
\begin{equation*}
    \int_\Omega |Du| + \int_{\partial\Omega} |u|\ \dH \leq \int_\Omega (z,Du) - \int_{\partial\Omega} u[z,\nu]\ \dH.
\end{equation*}
Noting that $(z,Du) = |Du|$ as measures in $\Omega$, the previous inequality takes to
\begin{equation*}
    \int_{\partial\Omega} \big(|u|+ u[z,\nu] \big) \ \dH \leq 0.
\end{equation*}

As $\|[z,\nu]\|_{L^\infty(\partial\Omega)} \leq \|z\|_{L^\infty(\Omega)^N} \leq 1$, we deduce that $|u|+u[z,\nu] (x) =0$ $\mathcal{H}^{N-1}$-a.e. on $\partial\Omega$, namely that \eqref{def_pairing} holds. This concludes the proof of Theorem \ref{th:SS}.    
\end{proof}

\section{An application to a concave-convex problem}
\label{sec:existence}
Let us  consider  the following Dirichlet problem 
\begin{equation}
	\label{eq:PbMain}
	\begin{cases}
		\dis -\Delta_1 u = \lambda h(u)f(x) + g(x,u) & \text{in}\;\Omega,\\
		u=0 & \text{on}\;\partial\Omega,
	\end{cases}
\end{equation}
where, $\lambda\geq 0$ is a parameter and $f\in L^{N,\infty}(\Omega)$ is positive. Furthermore, $h\colon [0,\infty)\mapsto [0,\infty]$ is assumed to be a continuous {finite outside the origin and possibly singular at zero (i.e.  $h(0)=\infty$)}, while $g\colon \Omega \times [0,\infty) \mapsto [0,\infty)$ is just a Carath\'eodory function with no imposed growth at infinity.
\medskip

Under suitable conditions, we show both existence for $\lambda$ small and non-existence for $\lambda$ large enough; in particular it is worth mention that, as we aim to be as general as possible,   these  two results do not share common  hypotheses (although they are extremely related) and they can be viewed as independent results. Nevertheless, we stress that both are valid for the model problem defined by $h(s)=s^{-\gamma}$ with $\gamma>0$ and $g(x,s) = s^p$ where $p>0$. In fact, in any case we show that the growth of $g(x,s)$ in $s$ at infinity does not play any role, so that more generally one can consider functions like $g(x,s) = \kappa(s)$ where $\kappa \colon [0, \infty) \to [0,\infty)$ is a continuous function satisfying $\kappa(0)=0$ and $\lim_{s\to\infty} \kappa(s) = \infty$. 

\medskip

First, let us specify what we mean by a solution to \eqref{eq:PbMain}.

\begin{defin}
	A  non-negative function $u\in BV(\Omega)\cap L^\infty(\Omega)$ is a solution to problem \eqref{eq:PbMain} if $h(u)f,  \,g(x,u) \in L^1_{\rm loc}(\Omega)$ and if there exists $z\in \mathcal{D}\mathcal{M}^\infty(\Omega)$ with $\|z\|_{L^\infty(\Omega)^N}\le 1$ such that
	\begin{gather}
		\label{def:distrp=1}
		-\operatorname{div}z = \lambda h(u)f(x) + g(x,u) \quad \text{ in } \mathcal{D}'(\Omega), \\
		\label{def:zp=1} (z,Du)=|Du| \quad \text{ as measures in } \Omega,\\
		\label{def:bordo}
        u(x)+u(x)[z,\nu](x) = 0
		\text{ for  $\mathcal{H}^{N-1}$-a.e. } x \in \partial\Omega.
	\end{gather}
\end{defin}

\begin{remark}
Note that we are asking the solution to globally belong to $BV(\Omega)$ even though~\eqref{eq:PbMain} may have a strong singular term (for instance, when $h(s)=s^{-\gamma}$ with $\gamma >1$). When it comes to strongly singular problems driven by the $p$-Laplacian operator, only solutions with local finite energy are guaranteed (see~\cite{BO}). In fact, in some cases these solutions have unbounded global energy even for smooth regular data $f$ (see~\cite{LM, OP2}). However, for the 1-Laplacian this is not the case as one can always find finite energy solutions independently of the value of $\gamma$ (see~\cite{MAOP} for further insights).
\triang
\end{remark}

We start by dealing with the existence theorem and stating its hypotheses. On $h$ we suppose that 
\begin{equation}\label{eq:hyp_h}
	\displaystyle \exists\;\gamma> 0\;\ \text{such that}\;\  h(s)\le \frac{1}{s^\gamma} \ \text{for all } s> 0 \text{ and }  h(0)=\infty,
\end{equation}
while we assume the existence of a continuous function $\kappa\colon [0,\infty)\mapsto [0,\infty)$ such that
\begin{equation}\label{eq:hyp_g} 
    g(x,s)\leq \kappa(s), \text{ for } \mathrm{a.e.}\ x\in\Omega,\ \forall s\geq 0 \ \text{ and } \ \kappa(0)=0.
\end{equation}
Without any restriction, $\kappa(s)$ can be assumed increasing and unbounded at infinity. 
\medskip

Our strategy for showing existence of solution is to approach~\eqref{eq:PbMain} directly with 1-Laplace problems by desingularizing $h(s)$ and truncating $f$ as 
\begin{equation}	\label{eq:hyp_l_incr0}
	h_\varepsilon(s) = h(s+\varepsilon) \quad \text{ and } \quad f_\varepsilon = T_{1/\varepsilon} (f) \ \forall \varepsilon>0.	
\end{equation}
The existence of solution to the approximated problems shall follow by Theorem~\ref{th:SS}, so that some monotonicity properties on the lower order term are needed. To this aim, we suppose that for every $\varepsilon>0$ and every $\lambda>0$ there exists an increasing Carathe\'{o}dory function $l_{\varepsilon,\lambda}\colon \Omega\times [0,\infty)\to \mathbb{R}$ such that 
\begin{equation}
	\label{eq:hyp_l_incr}
	\lambda h_\varepsilon(s)f_\varepsilon(x) + g(x,s) + l_{\varepsilon,\lambda}(x,s) \text{ is non-decreasing in } s\geq 0
\end{equation}
and
\begin{equation}
	\label{eq:hyp_l_LN}
	l_{\varepsilon,\lambda}(x,s)\in L^N(\Omega) \text{ for all } s>0 \text{ fixed}.
\end{equation}

Furthermore, it is not restrictive to assume
\begin{equation}
	\label{eq:hyp_l_inf}
	\lim_{s\to \infty} l_{\varepsilon,\lambda}(x,s) = \infty \text{ uniformly in } x \quad \text{and} \quad l_{\varepsilon,\lambda}(x,s) \geq 0,\ \forall s\geq 0.
\end{equation}

The existence of solutions to \eqref{eq:PbMain} can be stated as follows.
\begin{theorem} \label{th:CC_exist}
	Assume that $h(s)$ satisfies~\eqref{eq:hyp_h}, that $g(x,s)$ verifies~\eqref{eq:hyp_g} and that~\eqref{eq:hyp_l_incr}--\eqref{eq:hyp_l_inf} hold. Let $0<f\in L^{N,\infty}(\Omega)$. Then there exists $\overline{\lambda}>0$ such that for every $\lambda\leq\overline{\lambda}$ problem~\eqref{eq:PbMain} has a solution $u\in BV(\Omega)\cap L^\infty(\Omega)$.
\end{theorem}

\begin{remark}
As already said in Remark \ref{rembv},  assumption~\eqref{eq:hyp_l_incr} may be formulated in an equivalent way involving only the $BV$ regularity of $h_\varepsilon(s)$ and $g(x,s)$ in the $s$ variable. When we deal with more regular functions, namely when functions $h_\varepsilon(s)$ and $g(x,s)$ are absolutely continuous in $s\in (0,\infty)$ for $\ae$ $x\in \Omega$, one can explicitly define $l_{\varepsilon,\lambda}(x,s)$ as 
\[
l_{\varepsilon,\lambda}(x,s) = \int_0^s -\left(\lambda h'_\varepsilon(t)f_\varepsilon(x) + \frac{\partial g(x,t)}{\partial t}\right)^- \ \mathrm{d}t, \text{ for } \ae\ x\in \Omega \text{ and } s\geq 0,
\]
where $s^{-} = \min\{0,s\}$ denotes the negative part of $s$.
\triang
\end{remark}

\begin{remark}
We assume $h(0)=\infty$ in order to assure the positivity of the solution (see~\cite{DGOP}). When $h(0)<\infty$, the same existence result can be proved in an easier way, but in this case the solutions are not necessarily positive and may even be identically zero.
\triang
\end{remark}

\begin{remark}
This existence result can be extended to the case where $f$ is allowed to be zero in a set of positive measure. Here, the concept of solution has to be changed. The main difference is that the characteristic function $\chi_{\{u>0\}}$ plays a role in the distributional formulation~\eqref{def:distrp=1}. For an in-depth explanation of this phenomenon, we refer the reader to~\cite{DGS}.
\triang
\end{remark}

Before proving Theorem \ref{th:CC_exist} we focus our attention on the non-existence result. Here, the idea is to use a positive eigenfunction associated to the first eigenvalue of the 1-Laplacian as test function in~\eqref{eq:PbMain} to arrive at some contradiction when $\lambda$ is large enough.
\medskip

Concerning $h$, it is assumed to satisfy that
\begin{equation}\label{eq:hyp_h_nonex}
	h(s)>0 \text{ for all } s\ge 0. 
\end{equation}

We also suppose the existence of a ball $B\subset\subset \Omega$ such that
\begin{equation}\label{eq:hyp_g_nonex}
    f(x)\geq c\ \text{for }\mathrm{a.e.}\ x\in B \qquad \text{and} \qquad g(x,s)\geq \kappa(s) \text{ for } \mathrm{a.e.}\ x\in B,\ \forall s\geq 0,
\end{equation}
where $c$ is a positive constant and $\kappa\colon [0,\infty)\to [0,\infty)$ is a function such that $\lim_{s\to\infty} \kappa(s) = \infty$.

\begin{theorem} \label{th:CC_nonexist}
    Assume that both~\eqref{eq:hyp_h_nonex} and~\eqref{eq:hyp_g_nonex} hold. Then, there is some $\tilde\lambda>0$ such that problem~\eqref{eq:PbMain} has no solution for $\lambda\geq \tilde\lambda$.
\end{theorem}

\begin{remark}
Theorem \ref{th:CC_nonexist} still holds true if we replace hypothesis $\lim_{s\to\infty} \kappa(s) = \infty$ by the less restrictive assumption ${\liminf_{s\to \infty} \kappa(s) > \lambda_{1,B}}$, where $\lambda_{1,B}$ denotes the first eigenvalue of the 1-Laplacian in $B$.  
    \triang
\end{remark}

\subsection{Existence for $\lambda$ small} 
Let us consider the following problems
\begin{equation}
	\label{eq:PbApprox}
	\begin{cases}
		\dis -\Delta_1 u_\varepsilon = \lambda h_{\varepsilon}(u_\varepsilon)f_\varepsilon + g(x,u_\varepsilon) & \text{in}\;\Omega,\\
		u_\varepsilon=0 & \text{on}\;\partial\Omega,
	\end{cases}
\end{equation}
where $h_\varepsilon$ and $f_\varepsilon$ have been defined into \eqref{eq:hyp_l_incr0}. 
\medskip

As we mentioned, the existence of such $u_\varepsilon$ will follow as an application of Theorem~\ref{th:SS}, inspired by the strategy introduced in \cite{B} for the case of the Laplacian. In the sequel, as a subsolution to \eqref{eq:PbApprox}, we consider the non-negative and bounded solution to
\begin{equation}
	\label{eq:PbSubsol}
	\begin{cases}
		\dis -\Delta_1 w_\varepsilon = \lambda h_\varepsilon(w_\varepsilon)f_\varepsilon  & \text{in}\;\Omega,\\
		w_\varepsilon=0 & \text{on}\;\partial\Omega,
	\end{cases}
\end{equation} 
which has been found in \cite[Theorem~3.3]{DGOP}. In particular, let us highlight that no smallness assumption on the datum $f$ is required as $h$ degenerates at infinity. We will denote by $z_{w_\varepsilon}$ the associated vector field to $w_\varepsilon$.

As a supersolution to \eqref{eq:PbApprox} we consider 
\begin{equation}
	\label{eq:PbSupersol}
	\begin{cases}
		\dis -\Delta_1 v_\varepsilon  = \frac{\Lambda (f_\varepsilon + 1)}{(v_\varepsilon+\varepsilon)^\gamma}  & \text{in}\;\Omega,\\
		v_\varepsilon=0 & \text{on}\;\partial\Omega,
	\end{cases}
\end{equation} 
where $\Lambda>0$ is a constant that will be fixed later.
\medskip

Again, the existence of a unique non-negative solution $v_\varepsilon\in BV(\Omega)\cap L^\infty(\Omega)$ for~\eqref{eq:PbSupersol} is due to~\cite[Theorems~3.3 and 3.5]{DGOP}. The vector field associated to $v_\varepsilon$ will be denoted by $z_{v_\varepsilon}$. Next, we give a uniform $L^\infty(\Omega)$-bound for $v_\varepsilon$.

\begin{proposition}
\label{prop:Bound}
    Let $f\in L^{N,\infty}(\Omega)$ be a non-negative function. Then the solution $v_\varepsilon$ of~\eqref{eq:PbSupersol} verifies
    \begin{equation*}
        \|v_\varepsilon\|_{L^\infty(\Omega)} \leq (\tilde{\mathcal{S}}_1 \|f+1\|_{L^{N,\infty}(\Omega)})^\frac{1}{\gamma} \Lambda^\frac{1}{\gamma},\ \forall \varepsilon>0.
    \end{equation*}
\end{proposition}
\begin{proof}
Taking $G_k(v_\varepsilon)\in BV(\Omega)\cap L^\infty(\Omega)$ ($k>0$) as test function in~\eqref{eq:PbSupersol} we get, after using~\eqref{eq:Green}, that
\begin{equation}
    \label{eq:Pf_Prop_Bnd_1}
    \int_\Omega (z_{v_\varepsilon},D G_k(v_\varepsilon)) - \int_{\partial\Omega} G_k(v_\varepsilon)[z_{v_\varepsilon},\nu] \ \dH = \int_\Omega \frac{\Lambda (f_\varepsilon + 1)}{(v_\varepsilon+\varepsilon)^\gamma} G_k(v_\varepsilon).
\end{equation}
Observe that $v_\varepsilon[z_{v_\varepsilon},\nu] =- v_\varepsilon$ $\mathcal{H}^{N-1}$-a.e. on $\partial\Omega$ implies $G_k(v_\varepsilon)[z_{v_\varepsilon},\nu] =-G_k(v_\varepsilon)$ $\mathcal{H}^{N-1}$-a.e. on $\partial\Omega$ since $G_k$ is a Lipschitz function such that $G_k(s)s\geq 0$ for all $s\in\R$. Moreover, using that $(z_{v_\varepsilon}, DG_k(v_\varepsilon)) = |DG_k(v_\varepsilon)|$ as measures in $\Omega$ by Lemma~\ref{lem:Composition}, from~\eqref{eq:Pf_Prop_Bnd_1} we deduce 
\begin{equation*}
    \|G_k(v_\varepsilon) \|_{BV(\Omega)} = \int_\Omega |DG_k(v_\varepsilon)| + \int_{\partial\Omega} G_k(v_\varepsilon) \ \dH = \int_\Omega \frac{\Lambda (f_\varepsilon + 1)}{(v_\varepsilon+\varepsilon)^\gamma} G_k(v_\varepsilon).
\end{equation*}

Now, in the right-hand of the previous we use H\"{o}lder's and Sobolev's inequalities to obtain
\begin{equation*}
    \begin{split}
    \|G_k(v_\varepsilon) \|_{BV(\Omega)} &= \int_\Omega \frac{\Lambda (f_\varepsilon + 1)}{(v_\varepsilon+\varepsilon)^\gamma} G_k(v_\varepsilon) \leq \frac{\Lambda}{(k+\varepsilon)^\gamma} \int_\Omega (f+1) G_k(v_\varepsilon)\\
    &\leq \frac{\Lambda}{(k+\varepsilon)^\gamma} \|f+1\|_{L^{N,\infty}(\Omega)} \|G_k(v_\varepsilon)\|_{L^{\frac{N}{N-1},1}(\Omega)}
    \leq \frac{\Lambda \tilde{\mathcal{S}}_1}{(k+\varepsilon)^\gamma} \|f+1\|_{L^{N,\infty}(\Omega)}  \|G_k(v_\varepsilon) \|_{BV(\Omega)}.
    \end{split}
\end{equation*}

Then, we have that
\begin{equation}
    \label{eq:Pf_Prop_Bnd_2}
    \left( 1- \frac{\Lambda \tilde{\mathcal{S}}_1 \|f+1\|_{L^{N,\infty}(\Omega)}}{(k+\varepsilon)^\gamma} \right) \|G_k(v_\varepsilon) \|_{BV(\Omega)} \leq 0.
\end{equation}

Fixing $k=(\|f+1\|_{L^{N,\infty}(\Omega)} \tilde{\mathcal{S}}_1)^\frac{1}{\gamma} \Lambda^\frac{1}{\gamma}$, it is verified that
\[
1-\frac{\Lambda \tilde{\mathcal{S}}_1 \|f+1\|_{L^{N,\infty}(\Omega)}}{(k+\varepsilon)^\gamma} >0
\]
for every $\varepsilon>0$. Therefore, from~\eqref{eq:Pf_Prop_Bnd_2} we deduce that $\|G_k(v_\varepsilon)\|_{BV(\Omega)}=0$ for all $\varepsilon>0$, which means that $\|v_\varepsilon\|_{L^\infty(\Omega)} \leq k = (\tilde{\mathcal{S}}_1 \|f+1\|_{L^{N,\infty}(\Omega)})^\frac{1}{\gamma} \Lambda^\frac{1}{\gamma}$ for every $\varepsilon>0$. This concludes the proof.
\end{proof}

In the spirit of~\cite{B}, in the next result, we are able to fix $\Lambda$ so that $v_\varepsilon$ is a supersolution to~\eqref{eq:PbApprox} at least for $\lambda$ less than a constant $\overline{\lambda}$. We stress that both $\Lambda$ and $\overline{\lambda}$ are independent from $\varepsilon$ when $\varepsilon$ is small.

\begin{lemma}\label{lem:Supersol}
    Assume that $h$ satisfies~\eqref{eq:hyp_h} and that $g$ satisfies~\eqref{eq:hyp_g}. Let $f\in L^{N,\infty}(\Omega)$ be a non-negative function. Then there exists $\varepsilon_0>0$, $\Lambda>0$ and $\overline{\lambda}\in (0,\Lambda)$ such that for every $\varepsilon\in(0,\varepsilon_0)$ and for every $\lambda\leq \overline{\lambda}$ the solution $v_\varepsilon$ of~\eqref{eq:PbSupersol} is a supersolution to~\eqref{eq:PbApprox}.
\end{lemma}
\begin{proof}
In order to prove the lemma, it is enough to find $\varepsilon_0>0$, $\Lambda>0$ and $0<\overline{\lambda}<\Lambda$ such that
\[
\frac{\Lambda (f_\varepsilon+1)}{(v_\varepsilon+\varepsilon)^\gamma} \geq \lambda h_\varepsilon(v_\varepsilon) f_\varepsilon + g(x,v_\varepsilon)
\]
for every $\varepsilon\in(0,\varepsilon_0)$ and every $\lambda\leq \overline{\lambda}$. The previous is equivalent to show that
\[
\frac{\lambda h_\varepsilon(v_\varepsilon) (v_\varepsilon+\varepsilon)^\gamma f_\varepsilon}{f_\varepsilon +1} + \frac{g(x,v_\varepsilon) (v_\varepsilon+\varepsilon)^\gamma}{f_\varepsilon +1} \leq \Lambda.
\]
Let us do some computations on the left-hand of the previous. First, observe that $h_\varepsilon(v_\varepsilon)(v_\varepsilon+\varepsilon)^\gamma \leq 1$ by~\eqref{eq:hyp_h} and that $g(x,v_\varepsilon) \leq \kappa(v_\varepsilon)$ by~\eqref{eq:hyp_g}. Thus, recalling that $\kappa(s)$ is increasing and that $\|v_\varepsilon\|_{L^\infty(\Omega)} \leq C\Lambda^{\frac{1}{\gamma}}$ by Proposition~\ref{prop:Bound}, we obtain
\begin{align*}
    \frac{\lambda h_\varepsilon(v_\varepsilon)(v_\varepsilon+\varepsilon)^\gamma f_\varepsilon}{f_\varepsilon+1} + \frac{g(x,v_\varepsilon)(v_\varepsilon+\varepsilon)^\gamma}{f_\varepsilon + 1}
    & \leq \frac{\lambda f_\varepsilon}{f_\varepsilon +1} + \frac{ \kappa(v_\varepsilon) (v_\varepsilon+\varepsilon)^\gamma}{f_\varepsilon +1} 
    \leq \lambda + \kappa(v_\varepsilon)(v_\varepsilon+\varepsilon)^\gamma \\
    & \leq \lambda + \kappa \big(C\Lambda^\frac{1}{\gamma} \big) \big(C\Lambda^\frac{1}{\gamma} + \varepsilon \big)^\gamma
    \leq \lambda + C_1 \kappa \big(C\Lambda^\frac{1}{\gamma} \big) (C^\gamma\Lambda + \varepsilon^\gamma).
\end{align*}
Now, we would like to get $\lambda + C_1 \kappa \big(C\Lambda^\frac{1}{\gamma} \big) (C^\gamma\Lambda + \varepsilon^\gamma) \leq \Lambda$, i.e. $\lambda\leq \Lambda \Big(1-C_1 C^\gamma \kappa\big( C\Lambda^\frac{1}{\gamma}\big) \Big) -\varepsilon^\gamma C_1 \kappa \big(C\Lambda^\frac{1}{\gamma} \big)$. Observe that function $s \mapsto s\Big(1- C_1 C^\gamma \kappa \big(Cs^\frac{1}{\gamma}\big) \Big)$ is positive near the origin since $\kappa(0)=0$, so we can fix $\Lambda$ as the point where this function reaches its maximum and we can take $\overline{\lambda} < \Lambda \Big(1-C_1C^\gamma \kappa \big(C\Lambda^\frac{1}{\gamma} \big) \Big)$. Finally, since $\varepsilon^\gamma C_1 \kappa \big(C\Lambda^\frac{1}{\gamma} \big) \to 0$ as $\varepsilon\to 0^+$, there exists $\varepsilon_0>0$ such that
\[
\lambda \leq \Lambda \Big(1-C_1 C^\gamma \kappa\big( C\Lambda^\frac{1}{\gamma}\big) \Big) -\varepsilon^\gamma C_1 \kappa \big(C\Lambda^\frac{1}{\gamma} \big)
\]
for all $\varepsilon\in(0,\varepsilon_0)$ and all $\lambda\leq \overline{\lambda}$. This concludes the proof.
\end{proof}

We are ready to show that there exists a solution $u_\varepsilon$ to the approximated problems~\eqref{eq:PbApprox}.
\begin{proposition} \label{prop:Approx}
    Assume that $h$ satisfies~\eqref{eq:hyp_h}, that $g$ verifies~\eqref{eq:hyp_g} and that~\eqref{eq:hyp_l_incr}--\eqref{eq:hyp_l_inf} hold. Let $f\in L^{N,\infty}(\Omega)$ be a non-negative function. Then there exists $\varepsilon_0>0$, $\Lambda>0$ and $\overline{\lambda}\in (0,\Lambda)$ such that for every $\varepsilon\in(0,\varepsilon_0)$ and for every $\lambda\leq \overline{\lambda}$ there exists a non-negative solution $u_\varepsilon\in BV(\Omega)\cap L^\infty(\Omega)$ to \eqref{eq:PbApprox} such that $\|u_\varepsilon\|_{L^\infty(\Omega)} \leq C$, with $C>0$ not depending on $\varepsilon$.
\end{proposition}

\begin{proof}
Our aim is to apply Theorem~\ref{th:SS}. In what follows, we check that all its hypotheses are satisfied.

Let $\varepsilon_0>0$, $\Lambda>0$ and $\overline{\lambda}\in (0,\Lambda)$ be the ones given by Lemma~\ref{lem:Supersol}. We fix $\varepsilon\in(0,\varepsilon_0)$ and $\lambda \leq \overline\lambda$. Let $w_\varepsilon \in BV(\Omega)\cap L^\infty(\Omega)$ be the solution to~\eqref{eq:PbSubsol} found, as already mentioned, in \cite[Theorem~3.3]{DGOP} and let $v_\varepsilon \in BV(\Omega)\cap L^\infty(\Omega)$ be the solution to~\eqref{eq:PbSupersol}. Of course, $w_\varepsilon$ is a subsolution and, by Lemma~\ref{lem:Supersol}, $v_\varepsilon$ is a supersolution to \eqref{eq:PbApprox}.

Let us first show that $w_\varepsilon$ and  $v_\varepsilon$ are ordered in $\Omega$. Indeed $w_\varepsilon$ and $v_\varepsilon$ are two almost 1-harmonic functions satisfying, due to~\eqref{eq:hyp_h}, that
\begin{equation*}
\lambda h_\varepsilon(w_\varepsilon)f_\varepsilon < \frac{\Lambda (f_\varepsilon+ 1)}{(v_\varepsilon + \varepsilon)^\gamma}\ \ \text{in } \{w_\varepsilon > v_\varepsilon\},
\end{equation*}
which allows to apply Theorem~\ref{th:MP_Nondecr} in order to deduce that $w_\varepsilon\leq v_\varepsilon$ almost everywhere in $\Omega$.

On the other hand, since both $w_\varepsilon$ and $v_\varepsilon$ are in $L^\infty(\Omega)$, by~\eqref{eq:hyp_l_LN} we have that
\begin{gather*}
    l_{\varepsilon,\lambda}(x,w_\varepsilon),\ l_{\varepsilon,\lambda}(x,v_\varepsilon) \in L^N(\Omega).
\end{gather*}

Therefore, we can apply Theorem~\ref{th:SS} to assure the existence of a solution $u_\varepsilon \in BV(\Omega)\cap L^\infty(\Omega)$ of~\eqref{eq:PbApprox} such that $w_\varepsilon \leq u_\varepsilon \leq v_\varepsilon$ almost everywhere in $\Omega$. Lastly, by Proposition~\ref{prop:Bound} we conclude that $\|u_\varepsilon\|_{L^\infty(\Omega)} \leq C,\ \forall \varepsilon\in (0,\varepsilon_0)$. 
\end{proof}

Now we show that the sequence $u_\varepsilon$ of solutions of~\eqref{eq:PbApprox} given by Proposition~\ref{prop:Approx} tends to a solution to~\eqref{eq:PbMain}. Henceforth, we shall consider $\varepsilon\in (0,\varepsilon_0)$, where $\varepsilon_0$ is that of Proposition~\ref{prop:Approx}. We denote by $z_{u_{\varepsilon}}\in\DM(\Omega)$ the associated vector field to $u_\varepsilon$.
\medskip

The next result shows that sequence $u_\varepsilon$ is always bounded in $BV(\Omega)$ independently of the value $\gamma$. To prove that, we use the idea presented in~\cite{MAOP}.

\begin{lemma}  
Assume that $h$ satisfies~\eqref{eq:hyp_h}, that $g$ verifies~\eqref{eq:hyp_g} and that~\eqref{eq:hyp_l_incr}--\eqref{eq:hyp_l_inf} hold. Let $f\in L^{N,\infty}(\Omega)$ be a non-negative function. Let $\overline{\lambda}>0$ be the one given by Proposition~\ref{prop:Approx} and fix $\lambda\leq\overline{\lambda}$. Then, we have that $u_\varepsilon$ is bounded in $BV(\Omega)$ with respect to $\varepsilon$.
In particular, there exists $u\in BV(\Omega)\cap L^\infty(\Omega)$ such that, up to a subsequence, $u_\varepsilon\to u$ in $L^q(\Omega)$ for $1\leq q<\infty$ and $u_\varepsilon\to u$ almost everywhere in $\Omega$.
\end{lemma}

\begin{proof} Since the right-hand of~\eqref{eq:PbApprox} is non-negative, taking $(u_\varepsilon-\|u_\varepsilon\|_{L^\infty(\Omega)}) \in BV(\Omega)\cap L^\infty(\Omega)$ as test function in \eqref{eq:PbApprox}, one deduces
\begin{equation*}
	-\int_\Omega (u_\varepsilon-\|u_\varepsilon\|_{L^\infty(\Omega)}) \, \operatorname{div}z_{u_{\varepsilon}} \leq 0.
\end{equation*}
Applying~\eqref{eq:Green} one gets
\begin{equation*}
    \int_\Omega (z_{u_{\varepsilon}},D(u_\varepsilon-\|u_\varepsilon\|_{L^\infty(\Omega)})) - \int_{\partial\Omega} (u_\varepsilon-\|u_\varepsilon\|_{L^\infty(\Omega)})[z_{u_{\varepsilon}},\nu] \ \dH \leq 0.
\end{equation*}
Recalling that $(z_{u_{\varepsilon}},Du_\varepsilon)=|Du_\varepsilon|$ as measures in $\Omega$ and that $u_\varepsilon[z_{u_{\varepsilon}},\nu] = -u_\varepsilon$ $\mathcal{H}^{N-1}$-$\mathrm{a.e.}$ on $\partial\Omega$, we deduce
\begin{equation*}
    \int_\Omega |Du_\varepsilon| + \int_{\partial\Omega} u_\varepsilon \ \dH +  \int_{\partial\Omega} \|u_\varepsilon\|_{L^\infty(\Omega)} [z_{u_{\varepsilon}},\nu] \ \dH \leq 0.
\end{equation*}
Moreover from the previous and from $\|u_\varepsilon\|_{L^\infty(\Omega)} \leq C$, one gains
\begin{align*}
	\|u_\varepsilon\|_{BV(\Omega)} \leq -\int_{\partial\Omega} \|u_\varepsilon\|_{L^\infty(\Omega)} [z_{u_{\varepsilon}},\nu] \ \dH \leq \|u_\varepsilon\|_{L^\infty(\Omega)} \mathcal{H}^{N-1}(\partial\Omega) \leq C \mathcal{H}^{N-1}(\partial\Omega),
\end{align*}
where we have used that $\|[z_{u_{\varepsilon}},\nu]\|_{L^\infty(\partial\Omega)}\leq 1$. This concludes the proof.
\end{proof}

Now, we are in position to prove Theorem~\ref{th:CC_exist}.

\begin{proof}[Proof of Theorem~\ref{th:CC_exist}]
\mbox{}
Let $u_\varepsilon$ be the solution to \eqref{eq:PbApprox} whose existence has been proved in Proposition \ref{prop:Approx}. We fix $\lambda\leq\overline{\lambda}$. For the sake of clarity, we divide the proof into several steps. 

	\smallskip

    \textbf{Step 1.} Existence of the vector field $z$.

    As $\|z_{u_{\varepsilon}}\|_{L^\infty(\Omega)^N}\le 1$, then there exists $z\in L^\infty(\Omega)^N$ such that $z_{u_{\varepsilon}}\rightharpoonup z$ *-weakly in $L^\infty(\Omega)^N$. Since the norm is *-weakly lower semicontinuous, then we also have $\|z\|_{L^\infty(\Omega)^N}\leq 1$.

	\smallskip

    \textbf{Step 2.} Distributional formulation \eqref{def:distrp=1}.
    
    Taking $0\leq \varphi\in C_c^1(\Omega)$ as test function in the weak formulation of \eqref{eq:PbApprox}, one has
    \begin{equation}\label{eq:Proof_Main_0}
    	\int_\Omega z_{u_{\varepsilon}} \cdot \nabla \varphi = \int_\Omega \lambda h_{\varepsilon}(u_\varepsilon)f_\varepsilon \varphi + \int_\Omega g(x,u_\varepsilon) \varphi.     	
    \end{equation}
    Now we pass to the limit \eqref{eq:Proof_Main_0} as $\varepsilon\to 0$. Let observe that $z_{u_{\varepsilon}}\rightharpoonup z$ *-weakly in $L^\infty(\Omega)^N$; moreover the second term on the right-hand of \eqref{eq:Proof_Main_0} simply passes to the limit as $u_\varepsilon$ is bounded in $L^\infty(\Omega)$ with respect to $\varepsilon$ thanks to Proposition \ref{prop:Approx}. Then, using the Fatou Lemma in the first term on the right-hand of \eqref{eq:Proof_Main_0}, one yields to 
    \begin{equation}
    \label{eq:Proof_Main_1}
    \int_\Omega \lambda h(u)f\varphi \leq 
    \int_\Omega z \cdot \nabla \varphi - \int_\Omega g(x,u) \varphi.
    \end{equation}
    
Hence  $h(u)f\in L^1_\mathrm{loc}(\Omega)$ and  $\{u=0\}\subseteq \{f=0\}$ since $h(0)=\infty$. This allows to deduce that $u>0$ almost everywhere in $\Omega$ as $f>0$.
    Then, in order to prove \eqref{def:distrp=1}, remains to deal with the singular term. For $0\leq \varphi\in C_c^1(\Omega)$ we split it into two integrals, namely
    \begin{equation}
    \label{eq:Proof_Main_2}
    \int_\Omega \lambda h_{\varepsilon}(u_\varepsilon)f_\varepsilon\varphi = \int_{\{u_\varepsilon \leq \delta\}} \lambda h_{\varepsilon}(u_\varepsilon) f_\varepsilon\varphi + \int_{\{u_\varepsilon> \delta\}} \lambda h_{\varepsilon}(u_\varepsilon) f_\varepsilon\varphi.
    \end{equation}

    First, we control the singular term near zero, i.e.   the first integral on the right-hand of~\eqref{eq:Proof_Main_2}. To do that, we take $V_\delta(u_\varepsilon)\varphi\in BV(\Omega)\cap L^\infty(\Omega)$ where $0\leq \varphi\in C_c^1(\Omega)$ and $V_\delta$ is defined in \eqref{not:Vdelta}, as test function in the weak formulation of \eqref{eq:PbApprox}, obtaining after  applying \eqref{eq:Pairing}  that
    \[
    \int_\Omega (z_{u_{\varepsilon}}, DV_\delta(u_\varepsilon)) \varphi + \int_\Omega z_{u_{\varepsilon}} \cdot \nabla\varphi \, V_\delta(u_\varepsilon) = \int_\Omega \lambda h_{\varepsilon}(u_\varepsilon)f_\varepsilon V_\delta(u_\varepsilon)\varphi + \int_\Omega g(x,u_\varepsilon) V_\delta(u_\varepsilon)\varphi.
    \]
    Since $(z_{u_{\varepsilon}}, DV_\delta(u_\varepsilon)) = -|DV_\delta(u_\varepsilon)|$ as measures in $\Omega$ by Lemma~\ref{lem:Composition}, we get
    \[
    \int_{\{u_\varepsilon\leq \delta\}} \lambda h_{\varepsilon}(u_\varepsilon) f_\varepsilon\varphi \leq \int_\Omega z_{u_{\varepsilon}} \cdot \nabla\varphi \, V_\delta(u_\varepsilon) - \int_\Omega g(x,u_\varepsilon) V_\delta(u_\varepsilon)\varphi.
    \]
    Now, we take limits when $\varepsilon\to 0^+$. In the second integral we use that $z_{u_{\varepsilon}} \to z$ *-weakly in $L^\infty(\Omega)^N$ and that $V_\delta(u_\varepsilon)\to V_\delta(u)$ in $L^1(\Omega)$, and in the last integral we apply Lebesgue Theorem by using that $u_\varepsilon\to u$ $\mathrm{a.e.}$ in $\Omega$, that $u_\varepsilon$ is bounded in $L^\infty(\Omega)$ and hypothesis~\eqref{eq:hyp_g}. Thus, we have
    \[
    0\leq \limsup_{\varepsilon\to 0^+} \int_{\{u_\varepsilon \leq \delta\}} \lambda h_{\varepsilon}(u_\varepsilon) f_\varepsilon\varphi \leq \int_\Omega z \cdot \nabla\varphi \, V_\delta(u) - \int_\Omega g(x,u) V_\delta(u)\varphi.
    \]
    Then, as $\delta \to 0^+$, the Lebesgue Theorem implies
    \begin{equation}
    \label{eq:Proof_Main_3}
        0\leq \lim_{\delta\to 0^+} \limsup_{\varepsilon \to 0^+} \int_{\{u_\varepsilon\leq \delta\}} \lambda h_{\varepsilon}(u_\varepsilon) f_\varepsilon\varphi \leq \int_{\{u=0\}} z \cdot \nabla\varphi \,  - \int_{\{u=0\}} g(x,u) \varphi = 0,
    \end{equation}
    where the last equality is due to the fact that $u>0$ almost everywhere in $\Omega$.

It remains to deal with the second integral on the right-hand of \eqref{eq:Proof_Main_2}. Without loss of generality, we can always assume that $\delta \notin \{\eta>0:|\{u=\eta\}|>0\}$ since this set is at most countable. In particular, this implies that $\chi_{\{u_\varepsilon>\delta\}}$ converges to $\chi_{\{u>\delta\}}$ $\mathrm{a.e.}$ in $\Omega$ as $\varepsilon\to 0^+$. Moreover, we have both
    \[
    h_{\varepsilon}(u_\varepsilon) f_\varepsilon \chi_{\{u_\varepsilon >\delta\}} \varphi \leq \sup_{s\in (\delta,\infty)} h(s) f\varphi \in L^1(\Omega)
    \]
    and
    \[
    h(u) f \chi_{\{u>\delta\}} \varphi \leq h(u) f \varphi \overset{\eqref{eq:Proof_Main_1}}{\in}L^1(\Omega),
    \]
    so we can apply Lebesgue Theorem to deduce 
    \begin{equation}
    \label{eq:Proof_Main_4}
    \lim_{\delta\to 0^+} \lim_{\varepsilon\to 0^+} \int_{\{u_\varepsilon > \delta\}} \lambda h_{\varepsilon}(u_\varepsilon) f_\varepsilon\varphi = \int_{\{u > 0\}} \lambda h(u)f\varphi = \int_\Omega \lambda h(u)f\varphi
    \end{equation}
    since $u>0$. From~\eqref{eq:Proof_Main_2},~\eqref{eq:Proof_Main_3} and~\eqref{eq:Proof_Main_4} we obtain
    \begin{equation}
    \label{eq:Proof_Main_5}
    \lim_{\varepsilon\to 0^+} \int_\Omega \lambda h_{\varepsilon}(u_\varepsilon) f_\varepsilon\varphi = \int_\Omega \lambda h(u)f\varphi.
    \end{equation}
    Therefore, once that the previous is in force, we can pass to the limit \eqref{eq:Proof_Main_0} as $\varepsilon\to 0^+$ in order to get
    \begin{equation*}
    \int_\Omega z \cdot \nabla \varphi = \int_\Omega \lambda h(u)f\varphi + \int_\Omega g(x,u) \varphi,
    \end{equation*}
    for any $0\leq \varphi\in C_c^1(\Omega)$. Here, we have used that $z_{u_{\varepsilon}}\rightharpoonup z$ *-weakly in $L^\infty(\Omega)^N$, equation~\eqref{eq:Proof_Main_5} and that $\int_\Omega g(x,u_\varepsilon) \varphi\to \int_\Omega g(x,u) \varphi$ since $u_\varepsilon$ is bounded in $L^\infty(\Omega)$. It is immediate to see that this equality, which has been proved for non-negative $\varphi\in C_c^1(\Omega)$, is also true for any $\varphi\in C_c^1(\Omega)$.

	\smallskip

    \textbf{Step 3.} Test functions allowed.

    We have proved that 
    \[
    -\operatorname{div} z = \lambda h(u)f + g(x,u) \quad \text{in } \mathcal{D}'(\Omega)
    \]
    and that $\lambda h(u)f + g(x,u)\in L^1_{\mathrm{loc}}(\Omega)$ (see~\eqref{eq:Proof_Main_1} and recall that $u\in L^\infty(\Omega)$). Applying~\cite[Lemma~5.3]{DGOP} we deduce that $\lambda h(u)f + g(x,u)\in L^1(\Omega)$ and that
    \[
    -\int_\Omega \varphi\operatorname{div} z = \int_\Omega \lambda h(u)f\varphi + \int_\Omega g(x,u)\varphi,\ \forall \varphi\in BV(\Omega)\cap L^\infty(\Omega).
    \]

	\smallskip

    \textbf{Step 4.} Identification of $z$ given by \eqref{def:zp=1}.

    We can take $u_\varepsilon \varphi \in BV(\Omega)\cap L^\infty(\Omega)$ with $0\leq \varphi\in C_c^1(\Omega)$ as test function in~\eqref{eq:PbApprox} to deduce
    \[
    -\int_\Omega u_\varepsilon \varphi\operatorname{div} z_{u_{\varepsilon}} = \int_\Omega \lambda h_{\varepsilon}(u_\varepsilon) u_\varepsilon f_\varepsilon \varphi + \int_\Omega g(x,u_\varepsilon) u_\varepsilon \varphi.
    \]
    Using~\eqref{eq:Pairing} this is equivalent to
    \begin{equation*}
        \int_\Omega (z_{u_{\varepsilon}},Du_\varepsilon) \varphi + \int_\Omega u_\varepsilon\, z_{u_{\varepsilon}}\cdot \nabla\varphi = \int_\Omega \lambda h_{\varepsilon}(u_\varepsilon)u_\varepsilon f_\varepsilon \varphi + \int_\Omega g(x,u_\varepsilon) u_\varepsilon \varphi.  
    \end{equation*}
    
    We want to pass to the limit as $\varepsilon\to 0^+$. In the first term we take into account that $(z_{u_{\varepsilon}},Du_\varepsilon) = |Du_\varepsilon|$ as measures in $\Omega$ and then we use the lower semicontinuity with respect to the $L^1$-convergence of the functional $u\mapsto \int_\Omega \varphi|Du|$. In the second term we pass to the limit simply using that $u_\varepsilon\to u$ in $L^1(\Omega)$ and $z_{u_{\varepsilon}}\to z$ *-weakly in $L^\infty(\Omega)^N$. With respect to the third term, we use Scheff\'{e}'s Lemma to strengthen the convergence~\eqref{eq:Proof_Main_5} to a strong convergence in $L^1(\Omega)$ and thus one easily passes to the limit due to boundedness of $u_\varepsilon$ in $L^\infty(\Omega)$. Finally, one deals with the last term in a simple way. Then, one gets 
    \begin{equation}
    \label{eq:Proof_Main_10}
    \int_\Omega \varphi |Du| \leq -\int_\Omega uz\cdot \nabla \varphi + \int_\Omega \lambda h(u) u f \varphi + \int_\Omega g(x,u) u \varphi.
    \end{equation}

    By Step 3, we can take $u\varphi\in BV(\Omega)\cap L^\infty(\Omega)$ as test function in~\eqref{eq:PbMain} to deduce that
    \[
    -\int_\Omega \varphi u \operatorname{div} z = \int_\Omega \lambda h(u)u f\varphi + \int_\Omega g(x,u)u \varphi,
    \]
    which together with~\eqref{eq:Proof_Main_10} yields, using~\eqref{eq:Pairing}, to
    \[
    \int_\Omega \varphi |Du| \leq \int_\Omega (z,Du)\varphi,\ \forall 0\leq \varphi\in C_c^1(\Omega).
    \]
    Since the reverse inequality is immediate as $\|z\|_{L^\infty(\Omega)^N}\leq 1$, we have that
    \[
    \int_\Omega \varphi |Du| = \int_\Omega (z,Du)\varphi,\ \forall 0\leq \varphi\in C_c^1(\Omega).
    \]
    Therefore, we get
    \[
    \int_\Omega \varphi |Du| = \int_\Omega (z,Du)\varphi, \forall \varphi\in C_c^1(\Omega)
    \]
    and thus we conclude that $(z,Du)=|Du|$ as measures in $\Omega$.

	\smallskip
	
    \textbf{Step 5.} Boundary condition \eqref{def:bordo}.

    Let us define $\sigma:=\max\{1,\gamma\}$. We take  $u_\varepsilon^\sigma \in BV(\Omega) \cap L^\infty(\Omega)$ as test function in~\eqref{eq:PbApprox} to get
    \[
    -\int_\Omega u_\varepsilon^\sigma \operatorname{div} z_{u_{\varepsilon}} = \int_\Omega \lambda h_{\varepsilon}(u_\varepsilon) u_\varepsilon^\sigma f_\varepsilon + \int_\Omega g(x,u_\varepsilon) u_\varepsilon^\sigma.
    \]

    Taking into account that it follows from Lemma~\ref{lem:Composition} that $(z_{u_{\varepsilon}},u_\varepsilon^\sigma)=|Du_\varepsilon^\sigma|$ as measures in $\Omega$ and that $u_\varepsilon^\sigma [z_{u_{\varepsilon}},\nu] = -u_\varepsilon^\sigma$ $\mathcal{H}^{N-1}$-a.e. on $\partial\Omega$, we can apply~\eqref{eq:Green} to obtain
    \begin{equation}
    \label{eq:Proof_Main_11}
    \int_\Omega |Du_\varepsilon^\sigma| + \int_{\partial \Omega} u_\varepsilon^\sigma \ \dH = \int_\Omega \lambda h_{\varepsilon}(u_\varepsilon) u_\varepsilon^\sigma f_\varepsilon + \int_\Omega g(x,u_\varepsilon) u_\varepsilon^\sigma.
    \end{equation}

    The choice of $\sigma$ and the boundedness of $u_\varepsilon$ in $L^\infty(\Omega)$ allow us, taking into account~\eqref{eq:hyp_h} and~\eqref{eq:hyp_g}, to pass to the limit as $\varepsilon\to 0^+$ on the right-hand of~\eqref{eq:Proof_Main_11} using the Lebesgue Theorem. Using the lower semicontinuity with respect to the $L^1(\Omega)$-convergence of the $BV(\Omega)$-norm in the left-hand of~\eqref{eq:Proof_Main_11}, we deduce
    \[
    \int_\Omega |Du^\sigma| + \int_{\partial\Omega} u^\sigma \ \dH \leq \int_\Omega \lambda h(u)u^\sigma f + \int_\Omega g(x,u) u^\sigma = -\int_\Omega u^\sigma \operatorname{div}z = \int_\Omega (z,Du^\sigma) - \int_{\partial\Omega} u^\sigma[z,\nu] \ \dH.
    \]
    Using that $(z,Du^\sigma)=|Du^\sigma|$ as measures in $\Omega$, we get
    \[
    \int_{\partial\Omega} u^\sigma(1+[z,\nu]) \ \dH \leq 0.
    \]
    Since $\|[z,\nu]\|_{L^\infty(\partial\Omega)}\leq 1$, we deduce that $u^\sigma(1+[z,\nu])=0$ $\mathcal{H}^{N-1}$-a.e. on $\partial\Omega$ and, as $u\in BV(\Omega)$, we conclude that $u(1+[z,\nu])=0$ $\mathcal{H}^{N-1}$-a.e. on $\partial\Omega$.
\end{proof}

\subsection{Non-existence for $\lambda$ large}\mbox{}
\medskip

In order to prove the non-existence result, we will make use of the first eigenfunction of the 1-Laplacian which, recall, is briefly presented in Section \ref{sec:eigenvalue}. Here the key is to ``put'' the right-hand of~\eqref{eq:PbMain} over $\lambda_{1,B}$, with $\lambda_{1,B}$ being the first eigenvalue for the Dirichlet 1-Laplacian on some ball $B\subset\subset \Omega$. 

\begin{proof}[Proof of Theorem~\ref{th:CC_nonexist}]
    \mbox{}
    Let $B$ be the ball on which hypothesis~\eqref{eq:hyp_g_nonex} holds and let $\lambda_{1,B}$ be the first eigenvalue of the 1-Laplacian in $B$. We denote by $\phi$ a positive eigenfunction associated to $\lambda_{1,B}$ in $B$. Observe that $\phi$ satisfies~\eqref{eq:PbEigen} in $B$ and that $\phi$ is a positive constant in $B$, namely $\phi\equiv \alpha>0$ in $B$. Without changing the notation, we extend $\phi$ to $\Omega$ by assigning it the value 0 in $\Omega\setminus B$. Then, we have $\left.\phi\right|_{\partial\Omega}=0$.

    Arguing by contradiction, assume that there exists a solution $u$ of~\eqref{eq:PbMain} for some $\lambda>0$ with associated vector field $z_u$. Taking $\phi$ as test function in~\eqref{eq:PbMain}, we get
    \begin{equation}\label{eq:Proof_NonEx_1}
        -\int_\Omega \phi \operatorname{div}z_u = \int_\Omega \lambda h(u)f\phi + \int_\Omega g(x,u)\phi = \alpha \int_B \lambda h(u)f + \alpha \int_B g(x,u).
    \end{equation}

    Let us focus on the left-hand of~\eqref{eq:Proof_NonEx_1}. Taking into account that $(z_u,D\phi)\leq |D\phi|$ as measures in $\Omega$ and that $(z_\phi, D \phi)=|D\phi|$ as measures in $B$, where $z_\phi$ is the vector field associated to $\phi$, we obtain
    \begin{equation}
        \begin{split}\label{eq:Proof_NonEx_2}
        -\int_\Omega \phi \operatorname{div}z_u 
        &\overset{\eqref{eq:Green}}{=} \int_\Omega (z_u,D\phi) \leq \int_\Omega |D\phi| 
        = \int_B |D\phi| + \int_{\partial B} \phi \ \dH
        = \int_B (z_\phi,D\phi) + \int_{\partial B} \phi \ \dH \\
        &\overset{\eqref{eq:Green}}{=} -\int_B \phi \operatorname{div}z_\phi + \int_{\partial B} \phi[z_\phi,\nu] \ \dH + \int_{\partial B} \phi \ \dH
        \\
        &= -\int_B \phi \operatorname{div}z_\phi + \int_{\partial B} \phi(1+[z_\phi,\nu]) \ \dH.
        \end{split}
    \end{equation}
    On the one hand, since $\phi$ satisfies~\eqref{eq:PbEigen} and $\phi\equiv \alpha$ in $B$, we have $-\int_B \phi \operatorname{div} z_\phi = \alpha \lambda_{1,B} |B|$. On the other hand, we know that $\phi(1+[z_\phi,\nu])=0 \text{ $\mathcal{H}^{N-1}$-a.e. on }  \partial\Omega$. Thus, from~\eqref{eq:Proof_NonEx_2} we deduce that
    \begin{equation}\label{eq:Proof_NonEx_3}
        -\int_\Omega \phi \operatorname{div}z_u \leq \alpha \lambda_{1,B} |B|.
    \end{equation}
    With respect to the right-hand  of~\eqref{eq:Proof_NonEx_1}, using \eqref{eq:hyp_g_nonex}, we have that
    \begin{equation} \label{eq:Proof_NonEx_4}
        \alpha\int_B \lambda h(u)f + \alpha \int_B g(x,u) \geq \alpha \int_B (c\lambda h(u) + \kappa(u)).
    \end{equation}
    From~\eqref{eq:Proof_NonEx_1},~\eqref{eq:Proof_NonEx_3} and~\eqref{eq:Proof_NonEx_4}, we get that
    \begin{equation*}
        \alpha \lambda_{1,B} |B| \geq \alpha \int_B (c \lambda h(u) + \kappa(u)).
    \end{equation*}
    
    Since $\lim_{s\to \infty} \kappa(s) =\infty$, we know that there exists some $s_0>0$ such that $\kappa(s)> \lambda_{1,B}$ for every $s\geq s_0$. Moreover, by~\eqref{eq:hyp_h_nonex} we know that $h(s)$ is greater than a constant $c_0>0$ for every $s\in[0,s_0]$.  Thus, we deduce
    \begin{equation*}
        \begin{split}
        \lambda_{1,B} |B| &\geq \int_B (c\lambda h(u) + \kappa(u)) = \int_{B\cap \{u\leq s_0\}} (c\lambda h(u) + \kappa(u)) + \int_{B\cap \{u > s_0\}} (c\lambda h(u) + \kappa(u))\\
        &\geq c \int_{B\cap \{u\leq s_0\}} \lambda h(u) + \int_{B\cap \{u > s_0\}} \kappa(u) \geq c c_0 \lambda |B\cap\{u \leq s_0\}| + \lambda_{1,B} |B\cap\{u > s_0\}|.
        \end{split}
    \end{equation*}
    Since $u$ is almost everywhere finite, one can always fix $s_0$ large enough in order to have $|B\cap\{u \leq s_0\}|>0$. Taking $\tilde\lambda =\frac{\lambda_{1,B}}{c c_0}$, for every $\lambda>
    \tilde\lambda$ we have that
    \begin{equation*}
        \lambda_{1,B} |B| \geq c c_0 \lambda |B\cap\{u \leq s_0\}| + \lambda_{1,B} |B\cap\{u > s_0\}| > \lambda_{1,B} (|B\cap\{u \leq s_0\}| + |B\cap\{u > s_0\}|) = \lambda_{1,B} |B|
    \end{equation*}
    and this is a contradiction. Therefore, for $\lambda\geq\tilde\lambda$ no solution to~\eqref{eq:PbMain} can exist.
\end{proof}

\appendix

\section{A density result}
\label{sec:app1}

Here we state a density result proved in~\cite[Theorem~2]{Dem1}. For the sake of completeness, we include here its proof.  

\begin{proposition}
\label{prop:MP_Dens}
     Suppose that $u\in BV(\Omega)$ is almost 1-harmonic with associated vector field $z$. Then, there exists an associated vector field $\tilde z$ with $\operatorname{div} \tilde z = \operatorname{div} z$ and a sequence $u_\varepsilon$ with $u_\varepsilon\in W_0^{1,1+\varepsilon}(\Omega)$ such that
    \begin{enumerate}[i)]
        \item $u_\varepsilon\to  u$ in $L^{\frac{N}{N-1}}(\Omega)$,\\[-1.5mm]
        \item $\int_\Omega |\nabla u_\varepsilon|^{1+\varepsilon} \to \int_\Omega |Du| + \int_{\partial\Omega} |u|\ \dH$,\\[-1.5mm]
        \item $|\nabla u_\varepsilon|^{\varepsilon-1} \nabla u_\varepsilon \rightharpoonup \tilde z$ in $L^q(\Omega)$ for every $q<\infty$,\\[-1.5mm]
        \item $\operatorname{div}\big(|\nabla u_\varepsilon|^{\varepsilon-1} \nabla u_\varepsilon \big) \to \operatorname{div} \tilde z$ in $L^N(\Omega)$.
    \end{enumerate}
\end{proposition}
\begin{proof}
By density, one can take for every $\varepsilon>0$ a function $v_\varepsilon\in W_0^{1,1+\varepsilon}(\Omega)$ such that
\begin{equation}
    \label{eq:Proof_Dens_1}
    \left| \int_\Omega |\nabla v_\varepsilon|^{1+\varepsilon} - \int_\Omega |Du| - \int_{\partial\Omega} |u|\ \dH\right| < \varepsilon \quad \text{and} \quad \int_\Omega |v_\varepsilon-u|^{\frac{N}{N-1}} < \varepsilon.
\end{equation}
We denote $f:=-\operatorname{div} z \in L^N(\Omega)$ and we define
\begin{equation}
    \label{eq:Proof_Dens_2}
    \mu_\varepsilon := \inf_{v\in W_0^{1,1+\varepsilon}(\Omega)} \left\{ \frac{1}{1+\varepsilon}\int_\Omega |\nabla v|^{1+\varepsilon} + \frac{N-1}{N} \int_\Omega |v-u|^{\frac{N}{N-1}} - \int_\Omega f v \right\}.
\end{equation}

Since the functional associated to this minimizing problem is weakly lower semicontinuous and coercive, the infimum in~\eqref{eq:Proof_Dens_2} is attained at some $u_\varepsilon\in W_0^{1,1+\varepsilon}(\Omega)$. We will prove that $u_\varepsilon$ is a sequence that satisfies all the convergences appearing in the statement of the theorem.

First, using~\eqref{eq:Proof_Dens_1}, from~\eqref{eq:Proof_Dens_2} we get that
\begin{equation}
\begin{aligned}
\label{eq:Proof_Dens_3}
    \limsup_{\varepsilon\to0^+}\mu_\varepsilon &\leq \lim_{\varepsilon\to0^+} \left( \frac{1}{1+\varepsilon} \int_\Omega |\nabla v_\varepsilon|^{1+\varepsilon} + \frac{N-1}{N} \int_\Omega |v_\varepsilon-u|^{\frac{N}{N-1}} - \int_\Omega f v_\varepsilon \right) 
    \\
    &= \int_\Omega |Du| + \int_{\partial\Omega}|u| \ \dH - \int_\Omega f u = 0,
\end{aligned}
\end{equation}
where the last equality is due to
\[
\int_\Omega f u = -\int_\Omega u \operatorname{div} z \overset{\eqref{eq:Green}}{=} \int_\Omega (z,Du) - \int_{\partial \Omega} u[z,\nu] \ \dH = \int_\Omega |Du| + \int_{\partial\Omega}|u| \ \dH.
\]

Next, we prove that $u_\varepsilon$ is bounded in $BV(\Omega)$. Using Young's and H\"{o}lder's inequalities we obtain 
\begin{equation}
\begin{split}
\label{eq:Proof_Dens_4}
    \int_\Omega f u_\varepsilon = \int_\Omega f(u_\varepsilon - u) + \int_\Omega fu 
    & \leq \frac{N-1}{{2}N} \int_\Omega |u_\varepsilon-u|^{\frac{N}{N-1}} + \frac{2^{N-1}}{N} \|f\|^N_{L^N(\Omega)} + \|f\|_{L^N(\Omega)} \|u\|_{L^{\frac{N}{N-1}}(\Omega)}\\
    &:=  \frac{N-1}{{2}N} \int_\Omega |u_\varepsilon-u|^{\frac{N}{N-1}} + C.
\end{split}
\end{equation}
As $u_\varepsilon$ minimizes~\eqref{eq:Proof_Dens_2}, we get
\begin{equation}
\label{eq:Proof_Dens_5}
    \frac{1}{1+\varepsilon}\int_\Omega |\nabla u_\varepsilon|^{1+\varepsilon} + \frac{N-1}{N} \int_\Omega |u_\varepsilon-u|^{\frac{N}{N-1}} - \int_\Omega fu_\varepsilon = \mu_\varepsilon
\end{equation}
and, taking into account~\eqref{eq:Proof_Dens_4}, we have
\begin{equation}
\label{eq:Proof_Dens_8}
\frac{1}{1+\varepsilon}\int_\Omega |\nabla u_\varepsilon|^{1+\varepsilon} + \frac{N-1}{{2}N} \int_\Omega |u_\varepsilon-u|^{\frac{N}{N-1}} \leq \mu_\varepsilon + C.
\end{equation}
Since $\mu_\varepsilon$ is upper bounded by ~\eqref{eq:Proof_Dens_3}, this implies that $u_\varepsilon$ is bounded in $BV(\Omega)$. Then, $u_\varepsilon$ has a subsequence, not relabeled, such that $\nabla u_\varepsilon \to Dv$ *-weakly in $\mathcal{M}(\Omega)^N$ for some $v\in BV(\Omega)$; moreover, $u_\varepsilon \rightharpoonup v$ in $L^{\frac{N}{N-1}}(\Omega)$. Taking limits in~\eqref{eq:Proof_Dens_5}, by the lower semicontinuity of the $BV(\Omega)$ norm and by~\eqref{eq:Proof_Dens_3} we deduce
\begin{equation}
\label{eq:Proof_Dens_6}
    \int_\Omega |Dv| + \int_{\partial\Omega} |v| \ \dH + \frac{N-1}{N} \int_\Omega |v-u|^{\frac{N}{N-1}} - \int_\Omega fv \leq \limsup_{\varepsilon\to 0^+} \mu_\varepsilon\leq 0.
\end{equation}
Observe that one always has
\begin{equation}
\label{eq:Proof_Dens_7}
    \int_\Omega fv = -\int_\Omega v \operatorname{div}  z = \int_\Omega ( z,Dv) - \int_{\partial\Omega} v [z,\nu] \ \dH \leq \int_\Omega |Dv| + \int_{\partial\Omega} |v| \ \dH,
\end{equation}
where we have used~\eqref{eq:Green} and $\| z\|_{L^\infty(\Omega)^N}\leq 1$. Joining~\eqref{eq:Proof_Dens_6} and~\eqref{eq:Proof_Dens_7} we obtain, first, that $\int_\Omega |v-u|^{\frac{N}{N-1}}=0$ and thus $v=u$, and then that $\int_\Omega |\nabla u_\varepsilon|^{1+\varepsilon} \to \int_\Omega |Du| + \int_{\partial\Omega} |u| \ \dH$. This proves \textit{i)} and \textit{ii)}.

Now, we focus our attention on proving \textit{iii)}. Since $\into |\nabla u_\varepsilon|^{1+\varepsilon}$ is uniformly bounded by~\eqref{eq:Proof_Dens_8}, using the Holder inequality we have for every $1\leq q<\frac{1+\varepsilon}{\varepsilon}$ that
\begin{equation}
    \label{eq:Proof_Dens_9}
    \into \big| |\nabla u_\varepsilon|^{\varepsilon-1} \nabla u_\varepsilon \big|^q = \into |\nabla u_\varepsilon|^{\varepsilon q} \leq \left( \into |\nabla u_\varepsilon|^{1+\varepsilon} \right)^{\frac{\varepsilon q}{1+\varepsilon}} |\Omega|^{1-\frac{\varepsilon q}{1+\varepsilon}} \leq C^{\frac{\varepsilon q}{1+\varepsilon}} |\Omega|^{1-\frac{\varepsilon q}{1+\varepsilon}}.
\end{equation}
This implies that $|\nabla u_\varepsilon|^{\varepsilon-1}\nabla u_\varepsilon$ is bounded in $L^q(\Omega)^N$ with respect to $\varepsilon$. Then, there exists $\tilde z_q\in L^q(\Omega)^N$ such that, up to subsequences,
\[
|\nabla u_\varepsilon|^{\varepsilon-1}\nabla u_\varepsilon \rightharpoonup \tilde z_q \text{ weakly in }L^q(\Omega)^N.
\]
A standard diagonal argument shows that there exists a vector field $\tilde z$ independent of $q$ such that, up to subsequences,
\begin{equation}
    \label{eq:Proof_Dens_10}
    |\nabla u_\varepsilon|^{\varepsilon-1}\nabla u_\varepsilon \rightharpoonup \tilde z \text{ weakly in }L^q(\Omega)^N, \forall q<\infty.
\end{equation}
Taking limits as $\varepsilon\to 0^+$, it follows from the weak lower semicontinuity in~\eqref{eq:Proof_Dens_9} that
\[
\|\tilde z\|_{L^q(\Omega)^N} \leq |\Omega|^{\frac{1}{q}},
\]
and thus taking limits as $q\to\infty$ we get that $\|\tilde z\|_{L^\infty(\Omega)^N} \leq 1$.

In order to prove $iv)$, observe that since functions $u_\varepsilon$ are solutions to the minimization problem~\eqref{eq:Proof_Dens_2}, they satisfy (recall $f=-\operatorname{div}z$) the associated Dirichlet problem to the equation
\begin{equation}
\label{eq:Proof_Dens_11}
    -\operatorname{div}(|\nabla u_\varepsilon|^{\varepsilon-1} \nabla u_\varepsilon) + |u_\varepsilon - u|^{\frac{N}{N-1}-2} (u_\varepsilon - u) = -\operatorname{div}z.
\end{equation}
By~\eqref{eq:Proof_Dens_10} and since $u_\varepsilon\rightharpoonup u$ in $L^{\frac{N}{N-1}}(\Omega)$, one can easily pass to the limit in the weak formulation of~\eqref{eq:Proof_Dens_11} to get
\begin{equation}
\label{eq:Proof_Dens_12}
    -\operatorname{div}\tilde z = -\operatorname{div}z
\end{equation}
and, as $|u_\varepsilon - u|^{\frac{N}{N-1}-2} (u_\varepsilon - u)\to 0$ in $L^N(\Omega)$, one also has
\[
-\operatorname{div}(|\nabla u_\varepsilon|^{\varepsilon-1} \nabla u_\varepsilon) \to -\operatorname{div} z = -\operatorname{div} \tilde z \quad \text{ in } L^N(\Omega).
\]

It just remains to prove that $\tilde z$ is an associated vector field to $u$. First, given $0\leq \varphi\in C_c^1(\Omega)$ we multiply~\eqref{eq:Proof_Dens_11} by $u_\varepsilon \varphi$ getting that
\begin{equation*}
    \into |\nabla u_\varepsilon|^{1+\varepsilon}\varphi + \into |\nabla u_\varepsilon|^{\varepsilon-1} \nabla u_\varepsilon \nabla \varphi \ u_\varepsilon + \into |u_\varepsilon - u|^{\frac{N}{N-1}-2} (u_\varepsilon - u) u_\varepsilon \varphi = - \into u_\varepsilon\varphi \operatorname{div} \tilde z.
\end{equation*}
We can easily pass to the limit in the last three integrals and, using the lower semicontinuity in the first one, we obtain
\begin{equation*}
    \into |Du|\varphi + \into u \tilde z \cdot \nabla \varphi \leq - \into u\varphi \operatorname{div} \tilde z \overset{\eqref{eq:Pairing}}{=} \int_\Omega (\tilde z, Du) \varphi + \into u \tilde z \cdot \nabla \varphi.
\end{equation*}
In this way, $|Du|\leq (\tilde z,Du)$ as measures in $\Omega$, and, since $\|\tilde z\|_{L^\infty(\Omega)^N} \leq 1$, we conclude that $|Du| = (\tilde z,Du)$ as measures in $\Omega$. Finally, multiplying~\eqref{eq:Proof_Dens_12} by $u$ and using~\eqref{eq:Green}, we have
\[
\into (\tilde z, Du) + \int_{\partial \Omega} u[\tilde z,\nu] \ \dH = \into (z, Du) + \int_{\partial \Omega} u[ z,\nu] \ \dH.
\]
Since $\int_{\partial \Omega} u[z,\nu] \ \dH = - \int_{\partial \Omega}|u| \ \dH$ and $(\tilde z, Du) = (z,Du) = |Du|$ as measures in $\Omega$, we can say that
\[
\int_{\partial \Omega} \big( |u| + u[\tilde z,\nu] \big) \ \dH = 0
\]
and this implies $|u|+u[\tilde z,\nu]=0$ $\mathcal{H}^{N-1}$-a.e. on $\partial\Omega$ since $\|[\tilde z,\nu]\|_{L^\infty(\partial\Omega)}\leq 1$.
\end{proof}

\section{An auxiliary problem with an absorption term}
\label{sec:app2}

With the aim of providing a self-contained work, here we present with the existence of a solution to
\begin{equation}
    \label{eq:PbAbs}
	\begin{cases}
		\dis -\Delta_1 u + l(x,u) = f(x) & \text{in}\;\Omega,\\
		u=0 & \text{on}\;\partial\Omega,
	\end{cases}
\end{equation} 
with $f$ belonging to $L^N(\Omega)$ and $l\colon \Omega\times \R\to \R$ being a general Carathe\'{o}dory function. Let us emphasize that this existence result is the key for proving Theorem~\ref{th:SS}. As we will see, the presence of the absorption term has a regularizing effect, allowing the existence of solution without any smallness assumption on $f$ (to be comapred with~\cite{CT, MST1} where $l=0$ and with \cite{O} where an absorption term appears).
\medskip

More specifically, on the absorption term we assume that
\begin{equation}
    \label{eq:hyp_abs_l_inf}
    \lim_{s\to \pm \infty} l(x,s) = \pm \infty \text{ uniformly in } x \quad \text{and} \quad l(x,s)s \geq 0,\ \forall s\in\R
\end{equation}
and that for every $s_0>0$ one has  that
\begin{equation}
    \label{eq:hyp_abs_l_reg}
    \max_{s\in[-s_0,s_0]}|l(x,s)|\in L^{1}(\Omega).
\end{equation}

 Observe that, when $l(x,s)$ is non-decreasing, hypothesis~\eqref{eq:hyp_abs_l_reg} simply becomes $l(x,s)\in L^1(\Omega)$ for all $s\in \R$ fixed. To prove existence for~\eqref{eq:PbAbs}, the strategy is approximate by $p$-Laplacian type problems and then tend $p$ to 1. Therefore, for $p>1$ we consider the approximated problems
\begin{equation}
    \label{eq:PbAbsApprox}
	\begin{cases}
		\dis -\Delta_p u_p + l(x,u_p) = f & \text{in}\;\Omega,\\
		u_p=0 & \text{on}\;\partial\Omega.
	\end{cases}
\end{equation} 

By standard arguments, one can show that these problems~\eqref{eq:PbAbsApprox} have a solution $u_p\in W_0^{1,p}(\Omega)\cap L^\infty(\Omega)$. Indeed, one can proceed as follows. First, one approaches~\eqref{eq:PbAbsApprox} by truncating $l(x,s)$; the existence of solution in the $W_0^{1,p}(\Omega)$-space for these new approximated problems follows immediately by a fixed point argument (observe that $N>(p^*)'$). Then, Stampacchia's Theorem (see that $N>\frac{N}{p}$) gives a uniform $L^\infty(\Omega)$-bound that allows to pass to the limit thanks to~\eqref{eq:hyp_abs_l_reg}, getting rid of the truncation on $l(x,s)$ and obtaining thus a solution to~\eqref{eq:PbAbsApprox} in $W_0^{1,p}(\Omega)\cap L^\infty(\Omega)$.
\medskip

In the following we show the existence of a bounded solution to~\eqref{eq:PbAbs}. Since most of the proof is similar to the one of Theorem~\ref{th:SS}, we will only emphasize the points where the major differences appear.

\begin{theorem}
\label{th:Abs}
    Let $f\in L^N(\Omega)$ and assume that $l(x,s)$ verify~\eqref{eq:hyp_abs_l_inf} and~\eqref{eq:hyp_abs_l_reg}. Then, there exists a solution $u\in BV(\Omega) \cap L^\infty(\Omega)$ to problem~\eqref{eq:PbAbs} with $l(x,u)\in L^1(\Omega)$.
\end{theorem}

\begin{proof}
We start proving that $u_p$ is uniformly bounded in $L^\infty(\Omega)$.
We take $G_k(u_p)\in W_0^{1,p}(\Omega)\cap L^\infty(\Omega)$ as test function in~\eqref{eq:PbAbsApprox}, obtaining
\begin{equation}
    \label{eq:Proof_Abs_1}
    \int_\Omega |\nabla G_k(u_p)|^p + \int_\Omega l(x,u_p) G_k(u_p) = \int_\Omega f G_k(u_p).
\end{equation}

Let $h>0$ be a number that will be fixed latex. Due to hypothesis~\eqref{eq:hyp_abs_l_inf}, there exists some $\overline{k}>0$ such that
\begin{equation*}
    \inf_{|s|\in [k,\infty)} |l(x,s)| \geq h,\ \forall k\geq \overline{k}.
\end{equation*}
From now on, we will assume that $k\geq \overline{k}$. Then, for the second integral of~\eqref{eq:Proof_Abs_1} one has
\begin{equation}
    \label{eq:Proof_Abs_2}
    h \int_\Omega |G_k(u_p)| \leq \inf_{|s|\in [k,\infty)} |l(x,s)| \int_\Omega |G_k(u_p)| \leq \int_\Omega l(x,u_p) G_k(u_p).
\end{equation}

For the right-hand side of~\eqref{eq:Proof_Abs_1}, using H\"{o}lder and Sobolev inequalities we have
\begin{equation}
    \label{eq:Proof_Abs_3}
    \begin{split}
    \int_\Omega f G_k(u_p) &\leq \int_{\{|f|\leq h\}} |f| |G_k(u_p)| + \int_{\{|f|>h\}} |f| |G_k(u_p)| \\
    &\leq h \int_\Omega |G_k(u_p)| + \|f \chi_{\{|f|>h\}}\|_{L^N(\Omega)} \|G_k(u_p)\|_{L^{\frac{N}{N-1}}(\Omega)} \\
    &\leq h \int_\Omega |G_k(u_p)| + \|f \chi_{\{|f|>h\}}\|_{L^N(\Omega)} \mathcal{S}_1 \int_\Omega |\nabla G_k(u_p)|.
    \end{split}
\end{equation}

Joining~\eqref{eq:Proof_Abs_1},~\eqref{eq:Proof_Abs_2} and~\eqref{eq:Proof_Abs_3}, we deduce
\begin{equation}
    \label{eq:Proof_Abs_5}
    \int_\Omega |\nabla G_k(u_p)|^p \leq \|f \chi_{\{|f|>h\}}\|_{L^N(\Omega)} \mathcal{S}_1 \int_\Omega |\nabla G_k(u_p)|.
\end{equation}

On the other hand, using Young's inequality we get
\begin{align}
    \label{eq:Proof_Abs_4}
    \int_\Omega |\nabla G_k(u_p)| \leq \frac{1}{p}\int_\Omega |\nabla G_k(u_p)|^p + \frac{p-1}{p} |\{u_p>k\}| \leq \int_\Omega |\nabla G_k(u_p)|^p + |\{u_p>k\}|,
\end{align}
so~\eqref{eq:Proof_Abs_5} and~\eqref{eq:Proof_Abs_4} give
\begin{equation}
    \label{eq:Proof_Abs_7}
    \int_\Omega |\nabla G_k(u_p)| \leq \|f \chi_{\{|f|>h\}}\|_{L^N(\Omega)} \mathcal{S}_1 \int_\Omega |\nabla G_k(u_p)| + |\{u_p>k\}|.
\end{equation}

Now, we fix $h$ large enough in order to have
\begin{equation*}
    \|f \chi_{\{|f|>h\}} \|_{L^N(\Omega)} \mathcal{S}_1 < 1.
\end{equation*}

This allows to deduce from~\eqref{eq:Proof_Abs_7} that it holds
\begin{equation*}
    \int_\Omega |\nabla G_k(u_p)| \leq \frac{|\{u_p>k\}|}{1-\|f \chi_{\{|f|>h\}}\|_{L^N(\Omega)} \mathcal{S}_1},\ \forall k\geq \overline{k}.
\end{equation*}

Using Sobolev's inequality, for $l>k$ one obtains
\[
|l-k| \, |\{u_p>l\}|^\frac{N-1}{N} \leq
\left( \int_\Omega |G_k(u_p)|^\frac{N}{N-1} \right)^\frac{N-1}{N} \leq \frac{ \mathcal{S}_1 |\{u_p>k\}|}{1-\|f \chi_{\{|f|>h\}}\|_{L^N(\Omega)} \mathcal{S}_1},\ \forall l>k\geq \overline{k},
\]
which is equivalent to
\begin{equation}
    \label{eq:Proof_Abs_6}
     |\{u_p>l\}| \leq
    \left( \frac{\mathcal{S}_1}{1-\|f \chi_{\{|f|>h\}}\|_{L^N(\Omega)} \mathcal{S}_1} \right)^\frac{N}{N-1} \frac{|\{u_p>k\}|^\frac{N}{N-1}}{|l-k|^\frac{N}{N-1}}, \ \forall l>k\geq \overline{k}.
\end{equation}

Estimate~\eqref{eq:Proof_Abs_6} allows us to apply Stampacchia's Lemma to deduce that $u_p$ is uniformly bounded in $L^\infty(\Omega)$.

This uniform $L^\infty(\Omega)$-bound of $u_p$ allows us to obtain in a standard way a bound in $BV(\Omega)$. Indeed, taking  $u_p\in W_0^{1,p}(\Omega) $ as test function in~\eqref{eq:PbAbsApprox} and using the sign condition~\eqref{eq:hyp_abs_l_inf} of $l(x,s)$, one gets
\begin{equation*}
\int_\Omega |\nabla u_p|^p \leq \int_\Omega f u_p \leq C,
\end{equation*}
and Young's inequality implies
\begin{equation}
\label{eq:Proof_Abs_8}
\int_\Omega |\nabla u_p| \leq \frac{1}{p} \int_\Omega |\nabla u_p|^p + \frac{p-1}{p} |\Omega| \leq C+|\Omega|.
\end{equation}
The boundedness of $u_p$ in $BV(\Omega)$ follows immediately from~\eqref{eq:Proof_Abs_8} since $u_p$ is also bounded in $L^\infty(\Omega)$.

At this point, one can pass to the limit following the same steps as in Theorem~\ref{th:SS} (see also~\cite[Theorem~4.3]{MST1}).
\end{proof}

\begin{remark}

Observe that this proof can not be carried on when $f$ lies into the wider Marcinkiewicz space $L^{N,\infty}(\Omega)$. The reason is that, in general, $\|f \chi_{\{|f|>h\}} \|_{L^{N,\infty} (\Omega)}$ does not go to zero when $h$ goes to infinity. To see this, it suffices to consider the function $f(x)=|x|^{-1}$ in $\Omega = B_R(0)$. This function is in $L^{N,\infty}(\Omega) \setminus L^N(\Omega)$ and $\|f \chi_{\{|f|>h\}} \|_{L^{N,\infty} (\Omega)}= \omega_N^{1/N}$ for all $h>0$, where $\omega_N$ denotes the measure of the $N$-dimensional unit ball.

Then, one may wonder if this result can be extended for functions $f\in L^{N,\infty}(\Omega)$. The answer is negative, as the following example shows. For $N>2$, consider problem  
\begin{equation}
\label{eq:Prob_Example}
	\begin{cases}
		\dis -\Delta_1 u + u = \frac{M}{|x|} & \text{in}\; B_R(0),\\
		u=0 & \text{on}\;\partial B_R(0),
	\end{cases}
\end{equation} 
where $M>0$. Observe that solutions to this problem must be non-negative, so the boundary condition becomes either $u=0$ or $[z,\nu]=-1$ on $\partial\Omega$. We look for radial decreasing solutions, i.e., $u(x)=g(|x|)$ with $g'(s)<0$ for $0\leq s \leq R$. In this case, $z(x)=\frac{Du}{|Du|} = -\frac{x}{|x|}$ and $-\operatorname{div} z = \frac{N-1}{|x|}$, so the boundary condition is always fulfilled independently of the value of $u$ at the border. Then, the equation becomes
\[
-\operatorname{div}z + u = \frac{N-1}{|x|} + g(|x|) = \frac{M}{|x|}
\]
and thus $u(x) = g(|x|) = \frac{M-N+1}{|x|}$. Therefore, this unbounded function is the unique solution to~\eqref{eq:Prob_Example} in $BV(\Omega)$ when $M>N-1$ (recall that $u$ has to be radially decreasing). 

Finally, we stress that, since for $M<N-1$ function $u\equiv 0$ is a solution to $-\Delta_1 u =\frac{M}{|x|}$ in $B_R(0)$ with Dirichlet boundary conditions, then $u\equiv 0$ is the unique solution to~\eqref{eq:Prob_Example} (see~\cite[Remark~3.3]{CT} or~\cite[Theorem~3.2]{MST1}).
\triang
\end{remark}

\section*{Acknowledgements}
The first author has been funded by Junta de Andaluc\'ia (grant FQM-194),
by the Spanish Ministry of Science and Innovation, Agencia Estatal de Investigaci\'on (AEI)
and Fondo Europeo de Desarrollo Regional (FEDER) (grant PID2021-122122NB-I00) and by
the FPU predoctoral fellowship of the Spanish Ministry of Universities (FPU21/04849).
The second and the third author  have been  partially supported by the Gruppo Nazionale per l’Analisi Matematica, la Probabilità e le loro Applicazioni (GNAMPA) of the Istituto Nazionale di Alta Matematica (INdAM).

\section*{Conflict of interest declaration}

The authors declare no competing interests.

\section*{Data availability statement}
 We do not analyse or generate any datasets, because our work
proceeds within a theoretical and mathematical approach. One
can obtain the relevant materials from the references below.

\end{document}